\documentclass[runningheads,a4paper]{llncs}

\usepackage{cite}
\usepackage[cmex10]{amsmath}
\usepackage{amssymb}
\usepackage{url}
\usepackage[all]{xy}
\usepackage{latexsym}
\usepackage{xspace}

\input diagxy

\newcommand{\mvdash}[0]  {
	{\rightarrow}
}

\newcommand{\ran}[0]  {
  \mathit{Ran}
}
\newcommand{\disc}[0]  {
  \mathit{Disc}
}
\newcommand{\iterm}[0]  {%
\{\bullet\}
}
\newcommand{\iinit}[0]  {%
\{\}
}
\newcommand{\iprod}[0]  {%
\times
}
\newcommand{\icoprod}[0]  {%
\sqcup
}
\newcommand{\comma}[2]  {%
{#1\!\!\downarrow\!\!{#2}}
}

\newcommand{\catl}[1]  {
  \mathbb{#1}
}
\newcommand{\catw}[1]  {
  \mathbf{#1}
}
\newcommand{\onecat}[0]  {
  \catw{Cat}
}
\newcommand{\cat}[0]  {
  \catw{cat}
}
\newcommand{\word}[1]  {
  \mathit{#1}
}
\newcommand{\op}[1]  {
{{#1}^\word{op}}
}
\newcommand{\mor}[3]  {
  #1 \colon #2 \rightarrow #3
}
\newcommand{\tuple}[1]  {
	\langle #1 \rangle
}
\newcommand{\id}[1]  {
  \word{id}_{#1}
}
\newbox\anglebox
\setbox\anglebox=\hbox{\xy \POS(75,0)\ar@{-} (0,0) \ar@{-} (75,75)\endxy}

\newbox\aanglebox
\setbox\aanglebox=\hbox{\xy \POS(0,0)\ar@{-} (0,75) \ar@{-} (75,0)\endxy}

\begin{document}

\title{Logical systems I:\\Lambda calculi through discreteness}

\author{Michal R. Przybylek}
\institute{Faculty of Mathematics, Informatics and Mechanics\\
University of Warsaw\\
Poland}
\maketitle

\begin{abstract}
This paper shows how internal models for polymorphic lambda calculi arise in any 2-category with a notion of discreteness. We generalise to a 2-categorical setting the famous theorem of Peter Freyd saying that there are no sufficiently (co)complete non-degenerate categories. As a simple corollary, we obtain a variant of Freyd theorem for categories internal to any \emph{tensored} category. Also, with help of introduced concept of an associated category, we prove a representation theorem relating our internal models with well-studied fibrational models for polymorphism.
\end{abstract}

\section{Introduction}
\label{s:introduction}
The well-known Lambek-Curry-Howard isomorphism \cite{lambek} in its simplest form establishes a link between cartesian closed categories, simply typed lambda calculi and propositional intuitionistic logics:
\begin{center}
\begin{tabular}{c|c|c}
Category & $\lambda$-calculus & Logic\\
\hline
$1$ & $\{\bullet\}$ & $\top$\\
$A \times B$ & $A \times B$ & $A \wedge B$\\
$B^A$ & $A \rightarrow B$ & $A \Rightarrow B$\\
$0$ & $\emptyset$ & $\bot$\\
$A \sqcup B$ & $A \sqcup B$ & $A \vee B$\\
\end{tabular}
\end{center}
To a \emph{two}-category theorist, a category is just an object in a very well-behaved 2-category $\onecat$ of (locally small) categories. A natural question then is to ask what properties a 2-category has to posses to allow establishing the above connection inside the 2-category; and more importantly --- what can be gained by such considerations?

An open and still very active area of research in category theory is to give a reasonable characterisation of a 2-category that allows describing categorical constructions inside the 2-category. Some constructions like adjunctions, Kan extensions/liftings and fibrations/opfibrations \cite{jonpart} are easily definable in any 2-category. Others like pointwise Kan extensions/liftings require existence of particular finite limits. Some others like internal limits/colimits are much harder and require additional conditions or structures on the 2-category \cite{proarrows}\cite{proarrows2}\cite{yoneda}\cite{2topos}. In this paper we shall investigate \emph{internal} 2-categorical constructions through \emph{discreteness}. The following definition is standard.
\begin{definition}[Discreteness]\label{d:discreteness:1}
Let $\xy \morphism(0,0)|a|/@{>}@<3pt>/<300,0>[\catl{W}`\catl{C};U] \morphism(0,0)|b|/@{<-}@<-3pt>/<300,0>[\catl{W}`\catl{C};F]\endxy$ be an adjunction between categories $\catl{C}$ and $\catl{W}$ with $F$ left adjoint to $U$. This adjunction gives a notion of discreteness on category $\catl{W}$ if the unit of the adjunction is an isomorphism.
\end{definition}
Because the unit of an adjunction $F \dashv U$ is an isomorphism if and only if the left adjoint $F$ is fully faithful, we may identify $\catl{C}$ with the full image of $F$ and write $\mathit{Disc}_F(\catl{W})$ for it, dropping the subscript if $F$ is known from the context. The right adjoint to the inclusion will be usually denoted by $|{-}|$, so that for an object $A \in \catl{C}$ we have $U(F(A)) = |A|$, and the coreflection $|A| \rightarrow A$ (the counit of the adjunction) will be denoted by $\epsilon$.
One may find examples of discreteness.
\begin{example}[Discrete graph]\label{e:discrete:graph}
Let $\catw{Graph}$ be the category of undirected graphs and graph homomorphisms. Its full subcategory $\mathit{Disc}(\catw{Graph})$ consisting of graphs without edges gives a notion of discreteness on $\catw{Graph}$, with a discretisation functor $\mor{|{-}|}{\catw{Graph}}{\mathit{Disc}(\catw{Graph})}$ discarding all edges from a graph. Clearly, there is a natural isomorphism $\hom(D, G) \approx \hom(D, |G|)$, where $D$ is a discrete graph.
\end{example}
\begin{example}[Discrete topological space]\label{e:discrete:space}
Let $\catw{Top}$ be the category of topological spaces and continuous functions. Its full subcategory $\mathit{Disc}(\catw{Top})$ consisting of topological spaces for which every set is open, gives a notion of discreteness on $\catw{Top}$, with a discretisation functor $\mor{|{-}|}{\catw{Top}}{\mathit{Disc}(\catw{Top})}$ ``upgrading'' a topology on a space to the finest topology (i.e.~every set is open) on the space --- every function from a discrete space $D$ to any space $W$ is automatically continuous, since inverse image of any set is open in $D$; therefore, we have a natural isomorphism $\hom(D, W) \approx \hom(D, |W|)$.
\end{example}
A special care has to be taken in case $\catl{W}$ is a 2-category and $\catl{C}$ is a \mbox{1-category} --- here a notion of discreteness is induced by a \mbox{1-adjunction} ${F \vdash U}$ between underlying 1-categories with $F$ a 2-fully faithful functor; that is, there can be no non-trivial 2-morphisms in the full subcategory on the image of $F$.
\begin{definition}[Discreteness of a 2-category]\label{d:discreteness}
Let $\xy \morphism(0,0)|a|/@{>}@<3pt>/<300,0>[\catl{W}`\catl{C};U] \morphism(0,0)|b|/@{<-}@<-3pt>/<300,0>[\catl{W}`\catl{C};F]\endxy$ be a 1-adjunction between a 1-category $\catl{C}$ and a 2-category $\catl{W}$, where $F$ is a 2-fully faithful functor which is left 1-adjoint to a 1-functor $U$. This adjunction gives a notion of discreteness on category $\catl{W}$ if the unit of the adjunction is an isomorphism.
\end{definition}
\begin{example}[Discrete category]\label{e:discrete:category}
Let $\cat$ be the 2-category of small categories, functors and natural transformations. The category $\catw{Set}$ of sets and functions is its full subcategory inducing the notion of discreteness on $\cat$. The discretisation functor $\mor{|{-}|}{\cat}{\catw{Set}}$ sends a category to its underlying set of objects. The natural isomorphism:
$$\hom_{\cat}(X, \catl{C}) \approx \hom_\catw{Set}(X, |\catl{C}|)$$
follows directly from the definition of a functor. The situation generalises to any 2-category $\cat(\catl{C})$ of categories internal to a finitely complete category $\catl{C}$. Moreover, this situation also generalizes to any 2-category $\onecat{V}$ of categories enriched in a monoidal category $\catl{V}$ with initial object.
\end{example}
Although $\cat$ is a 2-category, we could not demand the inclusion \mbox{$\catw{Set} \rightarrow \cat$} to have right 2-adjoint --- clearly because there are no non-trivial 2-morphisms in a 1-category. In this example we could also characterise discrete categories $X$ as precisely these categories that satisfy the property: for every category $\catl{C}$ and every parallel functors $\mor{F, G}{\catl{C}}{X}$ there are no non-trivial (i.e.~other than identities) natural transformation $F \rightarrow G$. This suggests a very important generic notion of discreteness, which we shall call ``the canonical notion of discreteness''. 
\begin{definition}[Canonical discreteness]\label{d:discreteness:canonical}
Let $\catl{W}$ be a 2-category. Let us write $\mathit{Disc}(\catl{W})$ for the full subcategory of $\catl{W}$ consisting of these objects $X$, for which the category $\hom(C, X)$ is discrete in the sense of Example \ref{e:discrete:category} for every object $C \in \catl{W}$. We shall say that $\catl{W}$ has the canonical notion of discreteness if the inclusion $\mathit{Disc}(\catl{W}) \rightarrow \catl{W}$ has right 1-adjoint.
\end{definition}
Not every 2-category has the canonical notion of discreteness: consider the full 2-subcategory of $\cat$ consisting of all small categories excluding \emph{infinite} discrete categories. Clearly, the inclusion from the category $\catw{Set}_{\aleph_0}$ of finite sets and functions does not have a right adjoint.

Throughout the paper the concept of discreteness serves threefold purpose: in the next section it allows us to capture a good notion of internal cartesian closedness and a good notion of internal products, whereas in the third section it allows us to introduce the concept of an associated category.


Our first contribution is to extend the definition of fibred/internal connectives and polymorphism to an arbitrary 2-category with a notion of discreteness, and to show that a naive approach as in \cite{2topos} does not work properly. To justify that our proposed definitions give an appropriate extension, we provide a concept of an ``associated category''. This leads to our second contribution --- we show that with every finitely complete 2-category $\catl{W}$ that admits a notion of discreteness, one may associate a 2-functor realising $\catl{W}$ in a 2-category $\catw{Cat}(\disc(\catl{W}))$ of categories internal to the discrete objects of $\catl{W}$, in such a way that internal connectives and polymorphic objects are preserved. This realisation gives an equivalence of 2-categories if and only if discrete objects are dense. This sheds new light on the nature of fibred (co)products and their stability condition (i.e. the Beck-Chevalley condition). Moreover, because in the world of enriched categories discrete objects are not generally dense, we have to use our 2-categorical definitions since the usual fibrational definitions lose information about categories.
For the third contribution, we generalise the classical result of Freyd saying that the (co)completeness of a non-degenerate category have to be at a lower level on the set-theoretic hierarchy then the category itself, which is just another incarnation of Russel's paradox, Cantor's diagonal argument, Goedel's incompleteness theorem, or the result of Reynolds about non-existence of non-degenerate set-theoretic models for parametric polymorphism \cite{reynolds}. We show that if a 2-category is sufficiently rich, then its objects cannot have all internal products (therefore cannot be internally complete), unless are degenerated. Using the concept of an associated category, we obtain the Freyd theorem for categories internal to any \emph{tensored} category.

\section{Internal lambda calculi}
\label{s:models}
Let us recall that in any cartesian category $\catl{W}$ (i.e.~category with finite products) every object $A \in \catl{W}$ carries a unique comonoid structure $\xy \morphism(0,0)|a|/<-/<300,0>[1`A;!] \morphism(300,0)|a|/->/<400,0>[A`A \times A;\Delta]\endxy$, where $\Delta = \langle\mathit{id}, \mathit{id}\rangle$ is the diagonal morphism. In case $\catl{W} = \onecat$, we obtain the usual notion of terminal (initial) object and binary products (coproducts) in $\catl{A} \in \onecat$ by taking right (resp. left) adjoint to the comonoid structure on $\catl{A}$. It seems reasonable then, to internalise the notion of cartesian structure inside any cartesian 2-category $\catl{W}$ in the following way.
\begin{definition}[Internally (co)cartesian connectives]\label{d:internal:cc}
Let us assume that a 2-category $\catl{W}$ has finite products. An object $A \in \catl{W}$ has an \emph{internal} terminal value $\iterm_A$ (initial value $\iinit_A$) if the unique morphism $A \overset{!}{\rightarrow} 1$ has right adjoint ${1 \overset{\iterm_A}\rightarrow A}$ (resp.~left adjoint ${1 \overset{\iinit_A}\rightarrow A}$), and it has \emph{internal} products $\iprod_A$ (coproducts $\icoprod_A$) if the diagonal $A \overset{\Delta_A}\rightarrow A \times A$ has right adjoint ${A \times A \overset{\iprod_A}\rightarrow A}$ (resp.~left adjoint ${A \times A \overset{\icoprod_A}\rightarrow A}$).
\end{definition}
Yoneda lemma for 2-categories\footnote{Yoneda lemma for $\onecat$-enriched categories.} implies that for any (locally small) \mbox{2-category}  $\catl{W}$ the assignment:
$$A \in \catl{W} \mapsto \hom_\catl{W}(-, A) \in \onecat^{\catl{W}^{op}}$$
extends to a fully faithful 2-embedding:
$$\mor{y}{\catl{W}}{\onecat^{\catl{W}^{op}}}$$
called ``2-Yoneda functor''. Therefore, a morphism $f$ is adjoint to $g$ in $\catl{W}$ iff the transformation $\hom(-, f)$ is adjoint to the transformation $\hom(-, g)$ in $\onecat^{\catl{W}^{op}}$. Because 2-Yoneda functor also preserves finite products, it is possible to coherently give an external characterisation of \emph{internal} connectives in $\catl{W}$, even in case $\catl{W}$ does not have all finite products. Generally, we shall say that an object $A \in \catl{W}$ has a \emph{virtual} property, if its representable 2-functor $\mor{\hom(-, A)}{\catl{W}^{op}}{\onecat}$ has that property as an object in $\onecat^{\op{\catl{W}}}$. Thus, an object $A \in \catl{W}$ has a virtual internal terminal value (initial value, products, coproducts) if $\mor{\hom(-, A)}{\catl{W}^{op}}{\onecat}$ has internal terminal value (resp.~false value, products, coproducts) as an object in $\onecat^{\op{\catl{W}}}$.
The essence of virtual values is that although sometimes we may not have an access to the defining morphisms, there is always a natural assignment of parametrised values via universal properties. Recalling from \cite{barr} (Chapter I, Sections 4 and 5) the notion of generalised elements, let us write $\tau_X, \sigma_X \in A$ for morphisms $X \to A$, and then, given $\mor{s}{A}{B}$, $s(\tau_X) \in B$ for $s \circ \tau_X$. If an object $A \in \catl{W}$ has an internal virtual terminal value, then for every object $X \in \catl{W}$ there is a natural way to form a constant element $1_X \in A$ sending everything from $X$ to the virtual terminal value of $A$ --- it is given by the functor $\mor{({\iterm_{\hom(-, A)}})_X}{1}{\hom(X, A)}$ applied to the single object of the terminal category $1$. Similarly, given two generalised elements $\tau_X, \sigma_X \in A$  there is a canonical generalised element $\tau_X \times \sigma_X \in A$, provided $A$ has virtual internal products.

The definition of $\emph{internal}$ cartesian closedness is less obvious. One may pursue an approach of Mark Weber \cite{2topos} (Definition 8.1) and say that an object $A$ of a 2-category with finite products is internally cartesian closed if for every global element $1 \overset{x}\rightarrow A$
the morphism
$A \to^{\mathit{id}_A \times x} A \times A \overset{\iprod_A}\to A$
has a right adjoint.
Unfortunately, this definition is inadequate in various contexts --- including fibred and internal categories. For this reason we shall call the above concept ``naive cartesian closedness''.
\begin{example}[Failure of naive cartesian closedness]\label{e:failure:ccc}
A split indexed category is cartesian closed iff its every fibre is cartesian closed and reindexing morphisms preserve cartesian closed structure. Let $\catl{A}$ be a cartesian closed category for which there exists a category $\catl{X}$ such that $\catl{A}^\catl{X}$ is not cartesian closed\footnote{One may take for $\catl{A}$ any non-trivial free cartesian closed category, then $\catl{A}^{\{0 \rightarrow 1\}}$ is never cartesian closed.}. The 2-Yoneda functor gives an indexed category:
$$\mor{\hom(-, \catl{A})}{\onecat^{op}}{\onecat}$$
which is naively cartesian closed as an object in $\onecat^{\op{\onecat}}$. However, it is not a cartesian closed indexed category --- the fibre $\hom(\catl{X}, \catl{A}) = \catl{A}^\catl{X}$ over $\catl{X}$ is not cartesian closed. The problem with the naive definition is that choosing an element $\mor{x_1}{1}{\hom(1, A)}$
by naturality of $x$, chooses constant morphisms in every fibre. Therefore, naive cartesian closedness expresses existence of exponents of ``constant objects''. 
\end{example}
We shall generalise the idea of cartesian closedness provided by Bart Jacobs\footnote{We would get a proper generalisation if we substituted the notion of discreteness with the notion of ``grupoidalness''. Nonetheless, for the purpose of this paper it suffices to work with much simpler, yet not 2-categorical, concept of discreteness.} in Definition 3.9 in \cite{jacobspaper} for fibrations and adopt it to arbitrary cartesian 2-categories with a notion of discreteness\footnote{There is also a general notion of an internally closed object within $\star$-autonomous 2-categories (Definition 10 in \cite{hop}), however it cannot be generalised to our setting because cartesian 2-$\star$-autonomous categories are necessarily degenerated.}.
\begin{definition}[Internally closed connectives]\label{d:internal:ccc}
Let $\catl{W}$ be a cartesian 2-category with a notion of discreteness. An object $A \in \catl{W}$ is \emph{internally} cartesian closed if it has internall products and the morphism:
$$A \times |A| \to^{\langle \iprod_A \circ (\mathit{id} \times \epsilon_A), \pi_{|A|}\rangle} A \times |A|$$
has a right adjoint, where $\epsilon_A$ is the counit of the adjunction that gives the notion of discreteness on $\catl{W}$. 
\end{definition}
Let us see that this definition works for split indexed categories.
%
\begin{example}[Internally cartesian closed indexed category]\label{e:indexed:ccc}
A split indexed category $\mor{\Phi}{\op{\catl{C}}}{\onecat}$ is discrete in the sense of Definition \ref{d:discreteness:canonical} iff it is discrete in the usual sense --- i.e.~each of its fibres is a discrete category. Therefore $\Phi$ is a cartesian closed indexed category iff it is internally cartesian closed in the sense of Definition \ref{d:internal:ccc}.
\end{example}
%
We would like to extend the calculus of parametrised elements to internally closed connectives, but Example \ref{e:failure:ccc} shows that it is impossible in the full generality --- if $\catl{A}, \catl{X} \in \onecat$ are such that $\catl{A}$ is cartesian closed and $\catl{A}^\catl{X}$ is \emph{not} cartesian closed, then there is no way to form an exponent $\tau_\catl{X}^{\sigma_\catl{X}}$ for every pair of parametrised elements $\tau_\catl{X}, \sigma_\catl{X} \in \catl{A}$. However, this is possible if $\catl{X}$ is discrete. We shall postpone the proof of the following theorem until Section \ref{s:associated} (Theorem \ref{t:external:models}).
\begin{theorem}[Parametrised simply typed lambda calculus]\label{d:parametrised:lambda}
Let $\catl{W}$ be a cartesian 2-category with a notion of discreteness, and assume that an object $A \in \catl{W}$ is internally cartesian closed. Then for every discrete object $X \in \disc(\catl{W})$ the category $\hom_{\catl{W}}(X, A)$ is cartesian closed. Moreover, if $A$ is internally cocartesian (i.e~has an internal initial value and internal binary coproducts), then $\hom_{\catl{W}}(X, A)$ is cocartesian.
\end{theorem}
Therefore, an internally cartesian closed and cocartesian object $A \in \catl{W}$ for every discrete object $X \in \disc(\catl{W})$ gives a system of rules:
\begin{center}
\begin{tabular}{l@{\hskip 0.05in}l}
\begin{tabular}{c}
\\
\hline
$\tau_X\!\!\!\!\overset{\mathit{id}_{\tau_X}}{\vdash}\!\!\!\!\tau_X$ \\
\end{tabular} (id)
&
\begin{tabular}{c}
$\tau_X \overset{f}{\vdash} \sigma_X \;\; \sigma_X\overset{g}{\vdash} \rho_X$ \\
\hline
$\tau_X \overset{g \circ f}{\vdash} \rho_X$ \\
\end{tabular} (com)
\\[0.3in]
\begin{tabular}{c}
$x \in 1$\\
\hline
$\tau_X \overset{!}{\vdash} 1$ \\
\end{tabular} ($1$-int)
&
\begin{tabular}{c}
$x \in 1$\\
\hline
$0 \overset{*}{\vdash} \tau_X$ \\
\end{tabular} ($0$-int)
\\[0.3in]
\begin{tabular}{c}
$\rho_X \overset{f}{\vdash} \tau_X \;\; \rho_X\overset{g}{\vdash} \sigma_X$\\
\hline
$\rho_X\!\!\overset{\langle f, g\rangle}{\vdash}\!\!\tau_X \times \sigma_X$ \\
\end{tabular} ($\times$-int)
&
\begin{tabular}{c}
$\rho_X \overset{f}{\vdash} \tau_X \times \sigma_X$ \\
\hline
$\rho_X\!\!\!\!\!\!\overset{\pi_{\tau_X} \circ f}{\vdash}\!\!\!\!\!\!\tau_X \;\; \rho_X\!\!\!\!\!\!\overset{\pi_{\sigma_X} \circ f}{\vdash}\!\!\!\!\!\!\sigma_X$\\
\end{tabular} ($\times$-eli)
\\[0.3in]

\begin{tabular}{c}
$\tau_X \overset{f}{\vdash} \rho_X \;\; \sigma_X\overset{g}{\vdash} \rho_X$\\
\hline
$\tau_X \sqcup \sigma_X \!\!\overset{\lbrack f, g\rbrack}{\vdash}\!\!\rho_X$ \\
\end{tabular} ($\sqcup$-int)
&
\begin{tabular}{c}
$\tau_X \sqcup \sigma_X \overset{f}{\vdash}\rho_X$ \\
\hline
$\tau_X\!\!\!\!\!\!\overset{f \circ \iota_{\tau_X}}{\vdash}\!\!\!\!\!\!\rho_X \;\; \sigma_X\!\!\!\!\!\!\overset{f \circ \iota_{\sigma_X}}{\vdash}\!\!\!\!\!\!\rho_X$\\
\end{tabular} ($\sqcup$-eli)
\\[0.3in]

\begin{tabular}{c}
$\tau_X \times \sigma_X \overset{f}{\vdash} \rho_X$\\
\hline
$\tau_X\!\!\!\!\overset{\lambda y:\sigma_X f(-, y)}{\vdash}\!\!\!\!{\rho_X}^{\sigma_X}$ \\
\end{tabular} ($\lambda$-int)
&
\begin{tabular}{c}
$\tau_X\overset{f}{\vdash} {\rho_X}^{\sigma_X}$ \\
\hline
$\tau_X \times \sigma_X\!\!\!\!\overset{f(-)\cdot(=)}{\vdash}\!\!\!\!\rho_X$\\
\end{tabular} ($\lambda$-eli)
\\[0.3in]
\end{tabular}
\end{center}
which by Lambek-Curry-Howard isomorphism rises to a simply typed lambda calculus. 

More generally, given any morphism $\mor{r}{A \times A}{A}$, we shall say that an object $A$ is \emph{internally} left (resp.~right) $r$-closed if the morphism:\newline
${A \times |A| \to^{\langle r \circ (\mathit{id} \times \epsilon_A), \pi_{|A|}\rangle} A \times |A|}$ (resp.~${|A| \times A \to^{\langle r \circ (\epsilon_A \times \mathit{id}), \pi_{|A|}\rangle} A \times |A|}$) has a right adjoint. Following the terminology of Bourbaki we shall call an object $A$ together with a morphism $\mor{r}{A \times A}{A}$ a ``magma'', and internally $r$-left and $r$-right closed object a ``(bi)closed magma''. 
\begin{example}[Monoidal closed structure]\label{e:monoidal:cc}
A monoidal structure $\langle I, \otimes \rangle$ on a category $\catl{C}$ is left (resp. right) closed in the usual sense if it is internally left (resp. right) $\otimes$-closed.
\end{example}
\begin{example}[Lambek category]\label{e:lambek:cat}
Let us recall that a Lambek category is a category $\catl{C}$ together with a functor $\mor{R}{\catl{C} \times \catl{C}}{\catl{C}}$ such that for every object $A \in \catl{C}$ both $R(A, -)$ and $R(-, A)$ have right adjoints. A Lambek category is precisely a category which is an internally left and right $R$-closed magma.
\end{example}
We can go a bit further and define $r$-closedness in a general monoidal 2-category.
\begin{definition}[Internally closed connectives within a monoidal 2-category]\label{d:internal:mon:ccc}
Let $\catl{W}$ be a monoidal 2-category with a notion of discreteness such that its category of discrete objects $\disc(\catl{W})$ is cartesian and the embedding $\mor{F}{\disc(\catl{W})}{\catl{W}}$ is op-lax monoidal. An object $A \in \catl{W}$ together with a morphism $\mor{r}{A \otimes A}{A}$ is \emph{internally} left $r$-closed if:
$$A \otimes F(|A|) \to^{\!\!\!\id{} \otimes \theta \circ F(\Delta_{|A|})\!\!\!} A \otimes F(|A|) \otimes F(|A|) \to^{\!\!\!\id{} \otimes \epsilon \otimes \id{}\!\!\!} A \otimes A \otimes F(|A|) \to^{\!\!\!r \otimes \id{}\!\!\!} A \otimes F(|A|)$$
has a right adjoint, where $\mor{\epsilon}{F(|A|)}{A}$ is the counit of the adjunction that gives the notion of discreteness on $\catl{W}$, $\mor{\theta}{F(|A| \times |A|)}{F(|A|) \otimes F(|A|)}$ is the structure morphism from the definition of op-lax monoidal functor, and the natural isomorphisms expressing associativity of the tensor product $\otimes$ have been omitted for clarity.
Similarly, object $A$ is right $r$-closed if the morphism:
$$F(|A|) \otimes A \to^{\!\!\!\theta \circ F(\Delta_{|A|}) \otimes \id{} \!\!\!} F(|A|) \otimes F(|A|) \otimes A \to^{\!\!\!\id{} \otimes \epsilon \otimes \id{}\!\!\!} F(|A|) \otimes A \otimes A \to^{\!\!\!\id{} \otimes r\!\!\!} F(|A|) \otimes A$$
has right adjoint.
\end{definition}
\begin{example}[Enriched categories]\label{e:connectives:enriched}
Let $\catl{C}$ be a monoidal category with initial object $0$ preserved by the tensor product, $\mor{|{-}|}{\onecat(\catl{C})}{\catw{Set}}$ be the discretisation functor for the 2-category $\onecat(\catl{C})$ of $\catl{C}$-enriched categories, and $\mor{F}{\catw{Set}}{\onecat(\catl{C})}$ its left adjoint.   

Consider a $\catl{C}$-enriched category $\catl{A}$ together with a $\catl{C}$-enriched functor $\mor{R}{\catl{A} \otimes \catl{A}}{\catl{A}}$. The functor $A \otimes F(|A|) \rightarrow A \otimes F(|A|)$ from the definition of left $R$-closedness is given by:
$$\tuple{A, X} \mapsto \tuple{A, X, X} \mapsto \tuple{R(A, X), X}$$
By the definition of the tensor product for enriched categories:
$$\hom_{\catl{A}\otimes F(|\catl{A}|)}(\tuple{R(A, X), X}, \tuple{B, Y}) = \hom_{\catl{A}}(R(A, X), B) \otimes \hom_{F(|\catl{A}|)}(X, Y)$$
Let us assume that for every $X \in F(|\catl{A}|)$ the functor $R(-, X)$ has right adjoint $X \multimap (-)$. We claim that ${\tuple{B, Y} \mapsto \tuple{Y \multimap B, Y}}$ is right adjoint to ${\tuple{A, X} \mapsto \tuple{R(A, X), X}}$. Using again the definition of the tensor product of categories:
$$\hom_{\catl{A}\otimes F(|\catl{A}|)}(\tuple{A, X}, \tuple{Y \multimap B, Y}) = \hom_{\catl{A}}(A, Y \multimap B) \otimes \hom_{F(|\catl{A}|)}(X, Y)$$
Therefore we have to show:
$$\hom_{\catl{A}}(R(A, X), B) \otimes \hom_{F(|\catl{A}|)}(X, Y) \approx \hom_{\catl{A}}(A, Y \multimap B) \otimes \hom_{F(|\catl{A}|)}(X, Y)$$
Because $F(|\catl{A}|)$ is discrete, we can argue by cases. If $X \neq Y$, then by discreteness $\hom(X, Y) = 0$, and by preservation of initial object by the tensor:
$$\hom_{\catl{A}}(R(A, X), B) \otimes 0 \approx 0 \approx \hom_{\catl{A}}(A, Y \multimap B) \otimes 0$$
On the other hand, if $X=Y$, then the situation reduces to the adjunction between $R(-, X)$ and $X \multimap (-)$.
Hence, if $\catl{A}$ is left $R$-closed in the usual sense, it is left $R$-closed in the sense of Definition \ref{d:internal:mon:ccc}. To see that the converse is true as well, it suffices put $Y = X$ in the above formula.
A symmetric argument shows that $\catl{A}$ is right $R$-closed iff for every ${X \in F(|\catl{A}|)}$ the functor $R(X, -)$ has right adjoint.
\end{example} 
In case $\catl{W}$ is a cartesian 2-category and $\mor{r}{A \times A}{A}$ is the diagonal morphism, by universal properties of products, Definition \ref{d:internal:ccc} coincides with Definition \ref{d:internal:mon:ccc}.
\begin{example}[Topological spaces]\label{e:top:closed}
Although category of topological spaces is not cartesian closed, very many interesting topological spaces are exponentiable. In fact for a topological space $A$ there exists right adjoint to $\mor{- \times A}{\catw{Top}}{\catw{Top}}$ if and only if $A$ is a core-compact space \cite{isbellspace}, which means that the underlying locale of its open sets is continuous. One then may think that a restriction to the subcategory of topological spaces consisting of core-compact spaces could work. However, this again is not the case, because an exponent of two core-compact spaces need not be core-compact\footnote{An example of a subcategory of topological spaces that is cartesian closed is the category of compactly generated topological spaces \cite{top} \cite{ccspace}.}.
This example shows that sometimes we need even more general notion of internal closedness of one object with respect to another object. Formally, we shall say that given any morphisms $\mor{j}{B}{A}$ and $\mor{r}{A \times A}{A}$, an object $A$ is internally left (resp. right) $r$-closed with respect to ``the inclusion'' $j$ if:
$A \times |B| \to^{\langle r \circ (\mathit{id} \times j \circ \epsilon_B), \pi_{|B|}\rangle} A \times |B|$ (resp.~${|B| \times A \to^{\langle r \circ (j \circ \epsilon_B \times \mathit{id}), \pi_{|B|}\rangle} A \times |B|}$) has a right adjoint. According to this definition $\catw{Top}$ is cartesian closed with respect to the subcategory of core-compact spaces. 
\end{example}

We shall extend our lambda calculi by a notion of polymorphism. 
\begin{definition}[Parametrised (co)products]\label{d:param:prod}
Let $\catl{W}$ be a 2-category. Consider an object $A \in \catl{W}$, and a morphism $\mor{s}{X}{Y} \in \catl{W}$. A parametrised element $\tau_X \in A$ has a (co)product along $s$ if the right (resp.~left) Kan extension $\prod_s \tau_X$ (resp.~$\coprod_s \tau_X$) of $\tau_X$ along $s$ exists. That is, there is a morphism ${\mor{\prod_s \tau_X}{Y}{A}}$ (resp.~${\coprod_s \tau_X}$) and natural in $\mor{h}{Y}{A}$ bijections ${\hom(h, \prod_s \tau_X) \approx \hom(h \circ s, \tau_X)}$ (resp.~${\hom(\coprod_s \tau_X, h) \approx \hom(\tau_X, h \circ s)}$).

Moreover, we call the (co)product stable if the Kan extension is pointwise, meaning that the Kan extension is stable under comma objects. That is, for any diagram with a comma object square:
$$\bfig
\square(0, 0)<500,400>[\comma{i}{s}`X`I`Y;\pi_1`\pi_2`s`i]
\place(280, 200)[\twoar(1, 1)\scriptstyle{\pi}]

\place(680, 200)[\twoar(1, 1)\scriptstyle{\epsilon}]
\ptriangle(500, 0)/->``<-/<800, 400>[X`A`Y;\tau_X``\prod_s \tau_X ]
\efig$$
the composition: $$\epsilon \circ \pi_1 \bullet (\prod_s \tau_X) \circ \pi$$ exhibits $(\prod_s \tau_X) \circ i$ as the right Kan extension of $\tau_X \circ \pi_1$ along $\pi_2$; and dually, for any diagram with a comma object square:
$$\bfig
\square(0, 0)<500,400>[\comma{s}{i}`X`I`Y;\pi_1`\pi_2`s`i]
\place(250, 250)[\twoar(-1, -1)\scriptstyle{\;\pi}]
\place(680, 250)[\twoar(-1, -1)\scriptstyle{\;\eta}]
\ptriangle(500, 0)/->``<-/<800, 400>[X`A`Y;\tau_X``\coprod_s \tau_X]
\efig$$
the composition: $$(\coprod_s \tau_X) \circ \pi \bullet \eta \circ \pi_1$$ exhibits $(\coprod_s \tau_X) \circ i$ as the left Kan extension of $\tau_X \circ \pi_1$ along $\pi_2$.      
\end{definition}

\begin{example}[Internal (co)products]\label{e:prod:via:param}
Let $\catl{W}$ be a finitely complete 2-category with coproducts and $A \in \catl{W}$ an object with internal (co)products. Then for every object $X \in \catl{W}$ and every pair of parametrised elements $\tau_X, \sigma_X \in A$ the parametrised stable (co)product of cotuple $[\tau_X, \sigma_X]$ along the codiagonal $\mor{\nabla}{X \sqcup X}{X}$ exists
$$\bfig
\node xx(0, 1000)[X \sqcup X]
\node x(0, 700)[X]
\node a(900, 1000)[A]

\arrow/->/[xx`a;\lbrack\tau_X, \sigma_X \rbrack]
\arrow/->/[xx`x;\nabla]
\arrow|r|/->/[x`a;]
\place(150, 850)[\twoar(1, 1)]
\place(300, 900)[(\twoar(-1, -1))]
\place(500, 750)[\scriptstyle{\prod_\nabla \lbrack \tau_X, \sigma_X\rbrack}]
\place(500, 650)[\scriptstyle{(\coprod_\nabla \lbrack \tau_X, \sigma_X\rbrack)}]
\efig$$
and is equal to the internal (co)product $\tau_X \iprod_A \sigma_X$ (resp. $\tau_X \icoprod_A \sigma_X$). Indeed, by definition of Kan extensions we are looking for adjoint to:
$$\hom(X, A) \to^{(-) \circ \nabla} \hom(X \sqcup X, A) \approx \hom(X, A) \times \hom(X, A)$$
However by the universal property of an adjunction this morphism is isomorphic to the diagonal functor:
$$\hom(X, A) \to^\Delta \hom(X, A) \times \hom(X, A)$$
which by the usual 2-Yoneda argument has right (resp. left) adjoint since ${A \to^\Delta A \times A}$ does.
\end{example}
Let us elaborate on the stability condition. Given a diagram like in Definition \ref{d:param:prod}, we extend it by taking generalised elements $i_I \in Y, j_I \in X$ together with a generalised arrow $i_I \to^k  s(j_I)$, and form a comma object:
$$\bfig
\pullback(0, 0)[\comma{i_I}{s}`X`I`Y;\pi_1`\pi_2`s`i_I]/{@{>}@/^1em/}`-->`{@{>}@/_1em/}/[I;j_I``\mathit{id}]
\place(250, 250)[\twoar(1, 1)\scriptstyle{\pi}]
\place(700, 300)[\twoar(1, 1)\scriptstyle{\epsilon}]
\place(0, 690)[\twoar(1, 1)\scriptstyle{k}]
\ptriangle(500, 0)/->``<-/<500, 500>[X`A`Y;\tau_X``\prod_s \tau_X]
\efig$$
The stability condition tells us that we may define the product $\prod_s \tau_X $, which is a $Y$-indexed family, on each index $i_I \in Y$ separately by multiplying over generalised arrows $i_I \to  s(j_I)$, that is:
$$\{\prod_s \tau_X\}_{i_I} = \prod_{i_I \rightarrow  s(j_I)} \{\tau_{X}\}_{j_I}$$
In case $Y$ is canonically discrete, every line shrinks to a point, the comma object turns into pullback, and the above formula simplifies to:
$$\{\prod_s \tau_X\}_{i_I} = \prod_{i_I = s(j_I)} \{\tau_X\}_{j_I}$$
In the rest of the paper we shall mostly restrict to (co)products parametrised by discrete objects (restricting also the stability condition in Definition \ref{d:param:prod} to the subcategory of discrete objects), and call the (co)products polymorphic objects. Such polymorphism induces two additional rules for products:
\begin{center}
\begin{tabular}{l@{\hskip 0.1in}l}
\begin{tabular}{c}
$\left\{\sigma_{s(j)} \overset{f_{j}}{\mvdash} \tau_{j}\right\}_{j \in Y}$
\\
\hline
$\left\{\sigma_i\overset{\langle f_j \rangle_{i = s(j)}}{\mvdash}\underset{i = s(j)}{\prod}\tau_{j}\right\}_{i \in X}$
\\[0.2in]
($\prod$-int)\\
\end{tabular}
&
\begin{tabular}{c}
$\left\{\sigma_i \overset{f_i}{\mvdash}\underset{i = s(j)}{\prod}\tau_{j}\right\}_{i \in X}$ \\
\hline
$\left\{\sigma_{s(j)} \overset{\pi_j \circ f_{s(j)}}{\mvdash} \tau_{j}\right\}_{j \in Y}$\\[0.2in]
($\prod$-eli) \\
\end{tabular}
\end{tabular}
\end{center}
and dual for coproducts.
It is easiest to grasp the rules by the following example.
\begin{example}[Polymorphism in $\onecat$]\label{e:poly:cat}
Let $\onecat$ be the 2-category of locally small categories. Consider two sets $X, Y$ interpreted as categories in $\onecat$. A functor $\mor{F}{X \times Y}{\catl{C}}$ may be thought of as an $X, Y$-indexed family $\{\tau_{i, j}\}_{i \in X, j \in Y}$ of objects $\tau_{i, j} \in \catl{C}$, where $\tau_{i, j} = F(i, j)$. If $\catl{C}$ has $Y$-indexed products (in the usual sense), then with every such family, we may associate an $X$-indexed family $\{\prod_{j \in Y}\tau_{i,j}\}_{i \in X}$. Furthermore, this family satisfies the following universal property: for every $X$-indexed collection $\{\sigma_i\}_{i \in X}$ from $\catl{C}$ and every $X, Y$-indexed collection $\{\mor{f_{i, j}}{\sigma_i}{\tau_{i, j}}\}_{i\in X, j\in Y}$ of morphism from $\catl{C}$ there exists a unique collection of $X$-indexed morphisms $\{\mor{h_{i}}{\sigma_i}{\prod_{j\in Y}\tau_{i, j}}\}_{i\in X}$ from $\catl{C}$  such that $\pi_j^i \circ h_i = f_{i, j}$, where $\mor{\pi_j^i}{\prod_{j\in Y}\tau_{i, j}}{\tau_{i, j}}$ is the $j$-th projection of $i$-th element of the family. When $X$ is the singleton, the above reduces to ``internalisation'' of an external (that is set-indexed) collection of objects (types) $\{\tau_j\}_{j \in Y}$ into a single product object (type) $\prod_{j \in Y} \tau_j$.

In the above case, the product is taken along the cartesian projection $\mor{\pi_X}{X \times Y}{X}$. More generally, we may form a product along any function $\mor{s}{Z}{X}$ --- it assigns to a $Z$-indexed collection $\{\tau_{j}\}_{j \in Z}$ of objects $\tau_j \in \catl{C}$ an $X$-indexed collection $\{\prod_{i = s(j)}\tau_{j}\}_{i \in X}$.
\end{example}
The example shows that polymorphism in $\onecat$ is really an ``ad hoc polymorphism''. This is because every discrete category $\catl{X}$ is isomorphic to the coproduct over terminal category $\coprod_{|\catl{X}|}1$, and every morphism between discrete categories is induced by a function between indexes of the coproducts. Generally, we shall call such polymorphism ``ad hoc'' to stress the fact, that we are able to freely chose every element of the collection by choosing a generalised element on each of its components. It is better perhaps to think of $\coprod_\lambda \catl{A}$ as tensor of $\catl{A}$ with a discrete category $\lambda$. Here, we shall recall the notion of tensor in an arbitrary 2-category.
\begin{definition}[(co)Tensor]\label{d:tensor}
Let $\catl{W}$ be a 2-category, $A$ an object in $\catl{W}$, and $\lambda$ an ordinary small category. The tensor of $A$ with $\lambda$ exists, and is denoted by $\lambda \otimes  A$, if there exists a 2-natural isomorphism of 2-functors:
$$\hom_{\onecat}(\lambda, \hom_\catl{W}(A, -)) \approx \hom_{\catl{W}}(\lambda \otimes A, -)$$
Dually, the cotensor of $A$ with $\lambda$ exists, and is denoted by $\lambda \pitchfork A$, if there exists a 2-natural isomorphism of 2-functors:
$$\hom_{\onecat}(\lambda, \hom_\catl{W}(-, A)) \approx \hom_\catl{W}(-, \lambda \pitchfork A)$$
\end{definition}
If $\lambda$ is a set thought of as a discrete category, then the notion of tensor with $\lambda$ coincides with the coproduct over $\lambda$ --- clearly by the definition of a coproduct $\hom(\coprod_\lambda A, -) \approx \hom(\lambda, \hom(A, -))$ therefore $\lambda \otimes A \approx \coprod_\lambda A$. The usual codiagonal morphism $\mor{\nabla}{\coprod_\lambda A}{A}$ is the projection morphism $\mor{\pi}{\lambda \otimes A}{A}$ obtained via the transposition of the functor $\lambda \rightarrow \hom(A, A)$ sending everything from $\lambda$ to the identity on $A$. There is also a diagonal functor $\mor{\Delta}{\lambda}{\hom(A, \lambda \otimes A)}$ given by the transposition of the identity functor $\mor{\id{\lambda \otimes A}}{\lambda \otimes A}{\lambda \otimes A}$. Then every function between indexes $\mor{s}{\lambda'}{\lambda}$ induces a reindexing morphism $\mor{s \otimes A}{\lambda' \otimes A}{\lambda \otimes A}$, which is the transposition of $\mor{\Delta \circ s}{\lambda'}{\hom(A, \lambda \otimes A)}$. An ad hoc polymorphism is a polymorphism along such reindexing morphisms.  
\begin{definition}[Ad hoc polymorphism]\label{d:adhoc}
Let $A, X \in \catl{W}$ be two objects in a 2-category, and assume that the tensors $\lambda \otimes X$ and $\lambda' \otimes X$ with sets $\lambda$ and $\lambda'$ exist. An ad hoc $\lambda' \otimes X$-parametrised family $\mor{\tau}{\lambda' \otimes X}{A}$ has an ad hoc (co)product along a function $\mor{s}{\lambda'}{\lambda}$ if the parametrised (co)product of $\tau$ along the reindexing morphism $\mor{s \otimes X}{\lambda' \otimes X}{\lambda \otimes X}$ exists. In case the (co)product is taken over cartesian projection $\lambda \times \lambda' \rightarrow \lambda$ we write $\prod_{i \in \lambda'}\tau_i$ (resp. $\coprod_{i \in \lambda'}\tau_i$) for the ad hoc (co)product and call it ``simple (co)product''.
\end{definition}
The next example shows that in other 2-categories, other variants of polymorphisms are possible.
\begin{example}[Polymorphism in  $\cat(\omega\catw{Set})$]\label{e:omega:set}
Let $\omega\catw{Set}$ be the category whose objects are sets $X$ of pairs $\langle x, n \rangle$, where $n$ is a natural number, and whose morphisms $\mor{f}{X}{Y}$ are functions $\mor{f}{\pi_1[X]}{\pi_1[Y]}$ such that there exists a partially recursive function $e$ with the property: if $\tuple{x, n} \in X$ then ${\tuple{f(x), e(n)} \in Y}$. One may think of $\omega$-sets as of sets enhanced by ``proofs'' of the fact that elements belong to the set. Then a function between $\omega$-sets has to computably translate the proofs. In the above notation $\pi_1[-]$ is really a functor $\omega\catw{Set} \rightarrow \catw{Set}$ forgetting the proofs. Furthermore, it has right adjoint $\mor{F}{\catw{Set}}{\omega\catw{Set}}$ assigning to a set $X$ the $\omega$-set $\{\tuple{x, n} \colon x \in X, n \in N\}$, which means ``everything is a proof that an element belongs to the set for those elements that belong to the set'',  and making $\catw{Set}$ a reflective subcategory of $\omega\catw{Set}$. The category of $\omega$-sets has finite limits, therefore we may define the 2-category $\cat(\omega\catw{Set})$ of categories internal to $\omega\catw{Set}$.
%
We start with a definition of an ordinary category $\catw{PER}$ --- its objects are partial equivalence relations on the set of natural numbers, and its morphisms $\mor{f}{A}{B}$ from a PER $A$ to a PER $B$ are functions $\mor{f}{N/A}{N/B}$ between quotients of the relations, for which there exist partially recursive functions $e$ on natural numbers satisfying $f([a]_A) = [e(a)]_B$. One may think of category $\catw{PER}$ as realisation of Reynold's system $R$ \cite{rsys}\cite{plotkin}. A PER $A$ corresponds to a ``type''. Two elements $a, a'$ are ``the same'' from the perspective of type $A$ if $a A a'$, and an element $a$ belongs to type $A$ if $A$ recognises it, that is, if $a A a$. A function from a type $A$ to a type $B$ is thus a function between elements that maps ``the same'' elements to ``the same'' elements. We shall see that $\catw{PER}$ has also a natural $\omega$-set structure. First, let us observe that $\catw{PER}$ is cartesian closed --- a product of two PER's $A$ and $B$ is given by:
$$x (A\times B) y \Leftrightarrow \pi_1(x) A \pi_1(y) \wedge \pi_2(x) B \pi_2(y)$$
where $\mor{\pi_1, \pi_2}{(N \times N \approx N)}{N}$ are some chosen partially recursive projections, and the exponent is given by:
$$e B^A r \Leftrightarrow \forall_{a, a'} a A a' \Rightarrow e(a) B r(a')$$
under some chosen partially recursive enumeration of partially recursive functions. Therefore, $\catw{PER}$ may be thought of as a category enriched over itself. Then, observe that $\catw{PER}$ is a reflective subcategory of $\omega\catw{Set}$ --- the embedding $\catw{PER} \rightarrow \omega\catw{Set}$ sends a PER $A$ to the $\omega$-set of quotients:
$$\{\tuple{[n]_A, n} \colon n A n\}$$
and its left adjoint identifies elements along their proofs --- it sends an $\omega$-set $X$ to the relation $\widehat{X}$:
$$n \widehat{X} m \Leftrightarrow \exists_{\tuple{x, n}, \tuple{x', m} \in X}\; x \cong x'$$ where two elements belong to the same equivalence class of equivalence relation $\cong$ if they share a common proof: that is, $\cong$ is generated by $x \cong x'$, such that $\tuple{x, e} \in X$ and $\tuple{x', e} \in X$ for some $e$. Therefore, $\catw{PER}$ may be thought of as a category enriched over $\omega\catw{Set}$. Finally, observe that we may glue $\hom$-$\omega$-sets of such enriched category into a single $\omega$-set:
$$\catw{PER}_1 = \{\tuple{\langle A, B, [n]_{B^A} \rangle, n} \colon A, B \;\textit{are PER's and $n B^A n$}\}$$
making $\catw{PER}$ an $\omega\catw{Set}$-internal category. Now, if $X$ is an ordinary set, then $\omega$-functors (i.e.~$\omega\catw{Set}$-internal functors) $\mor{\tau_X, \sigma_X}{X}{\catw{PER}}$ are ordinary families of PERs. However, an $\omega$-natural transformation (i.e.~$\omega\catw{Set}$-internal natural transformation) $\mor{\alpha}{\tau_X}{\sigma_X}$ has to satisfy a uniformity condition: $$\bigcap_{x \in X} \alpha(x) \neq \emptyset$$
This means that $\mor{\alpha}{\tau_X}{\sigma_X}$ is determined by a single partially recursive function $\mor{e}{N}{N}$ such that for all $x \in X$ we have $a \tau_X(x) a' \Rightarrow e(a) \sigma_X(x) e(a')$. Therefore, the parametrised product of $\sigma_X$ is given by $\bigcap_{x \in X} \sigma_X(x)$:
$$\bfig
\node a(0, 1000)[\tau]
\node b(600, 700)[\sigma_X(x)]
\node prod(0, 700)[\bigcap_{x \in X} \sigma_X(x)]

\arrow/->/[a`b;\lbrack e \rbrack]
\arrow/->/[prod`b;\pi_x]
\arrow/-->/[a`prod;\exists ! \lbrack e \rbrack]
\efig$$
The projections $\bigcap_{x \in X} \sigma_X(x) \to^{\pi_x} \sigma_X(x)$ are induced by the identity function. For every constant $\omega$-functor $\mor{\tau}{X}{\catw{PER}}$, an $\omega$-natural transformation $\tau \rightarrow \sigma_X$ is induced by $e$ satisfying $\forall_{x \in X} a \tau a' \rightarrow e(a) \sigma_X(x) e(a')$. The last condition is equivalent to $a \tau a' \rightarrow e(a) (\bigcap_{x \in X}\sigma_X(x)) e(a')$. Therefore, every $\omega$-natural transformation $\tau \rightarrow \sigma_X$ uniquely determines a morphism $\tau \to \bigcap_{x \in X}\sigma_X(x)$. One may find that such products reassemble usual rules for intersection types in lambda calculi:
\begin{center}
\begin{tabular}{l@{\hskip 0.1in}l}
\begin{tabular}{c}
$\tau \overset{f}{\mvdash} \sigma_X$\\
\hline
$\tau \overset{f}{\mvdash} \bigcap_{x \in X} \sigma_X(x)$ \\
\end{tabular} ($\bigcap$-int)
&
\begin{tabular}{c}
$\tau \overset{f}{\mvdash} \bigcap_{x \in X} \sigma_X(x)$ \\
\hline
$\tau \overset{f}{\mvdash} \sigma_X$\\
\end{tabular} ($\bigcap$-eli)
\end{tabular}
\end{center}
By similar considerations, we get that the parametrised coproduct of $\sigma_X$ is $\bigcup_{x \in X} \sigma_X(x)$. An extension of Example \ref{e:prod:via:param} shows that internal (finite) (co)products may be obtained by using tensors $X \otimes 1$ in parametrisation instead of $X$. There is also an intermediate construction between $X$ and $X \otimes 1$ that yields uniform quantifiers. We may reach this construction by parameterising a category via the internal natural number object ${N_\omega = \{\langle n, n\rangle \colon n \in N\}}$ in $\omega\catw{Set}$. An  $N_\omega$-parametrised collection of objects from $\catw{PER}$ is any countable collection $\sigma(n)_{n \in N}$ of PER's. A product $\prod_{n \in N} \sigma_X(n)$, which in this context may be denoted by $\forall_{n \in N} \sigma_X(n)$, consists of partially recursive functions $e$ which applied to the $n$-th index return an element of $\sigma_X(n)$, that is: $e (\forall_{n \in N} \sigma_X(n)) r \Leftrightarrow \forall_{n \in N} \; e(n) \sigma_X(n) r(n)$. It should be noted that the last construction reduces to the usual dependent product in the ordinary category $\catw{PER}$ since the internal natural number object in $\catw{PER}$ is the same as the internal number object in $\omega\catw{Set}$.
%
%
%
%
\end{example}
If $\mor{\tau}{X}{C}$ is an $X$-parametrised element of $C$, then one may try to compute the parametrised product of $\tau$ along \emph{itself}:
$$\bfig
\node x(0, 1000)[X]
\node c1(600, 1000)[C]
\node c2(0, 700)[C]

\arrow/->/[x`c1;\tau]
\arrow/->/[x`c2;\tau]
\arrow|r|/->/[c2`c1;\prod_{\tau} \tau]

\place(230, 860)[\twoar(1, 1)]
\efig$$
\begin{definition}[(Density (co)product]\label{d:density:product}
A density (co)product $\mor{T_\tau}{C}{C}$ (resp. $\mor{D_\tau}{C}{C}$) of a parametrised element $\mor{\tau}{X}{C}$ is the (co)product of $\mor{\tau}{X}{C}$ along itself.
\end{definition} 
\begin{example}[Logical consequence]\label{e:logical:consequence:codensity}
Let $\onecat(2)$ be the 2-category of categories enriched in a 2-valued Boolean algebra $2 = \{0 \rightarrow 1\}$. A $2$-enriched category is tantamount to a partially ordered set (poset), and a $2$-enriched functor is essentially a monotonic function between posets. Let us consider a relation:
$${\models} \subseteq {\word{Mod} \times \word{Sen}}$$
thought of as a satisfaction relation between a set of models $\word{Mod}$ and a set of sentences $\word{Sen}$. By transposition, relation $\models$ yields the ``theory'' function $\mor{\word{th}}{\word{Mod}}{2^\word{Sen}}$, where $2^\word{Sen}$ is the poset of functions $\word{Sen} \rightarrow 2$, or equivalently the poset of subsets of $\word{Sen}$ .

Since ``power'' posets $2^\word{Sen}$ are internally complete in the 2-category $\onecat(2)$, the stable density product of $\mor{\word{th}}{\word{Mod}}{2^\word{Sen}}$ exists:
$$\bfig
\node x(0, 1000)[\word{Mod}]
\node c1(600, 1000)[2^\word{Sen}]
\node c2(0, 700)[2^\word{Sen}]

\arrow/->/[x`c1;\word{th}]
\arrow/->/[x`c2;\word{th}]
\arrow|r|/->/[c2`c1;T_\word{th}]
\place(230, 860)[\twoar(1, 1)]
\efig$$
and is given by the 2-enriched end:
$$T_\word{th}(\Gamma)(\psi) = \int_{M\in \mathit{Mod}} \word{th}(M)(\psi)^{\hom(\Gamma, \word{th}(M)(-))}$$
where $\psi \in \word{Sen}$ is a sentence, and $\Gamma \in 2^\word{Sen}$ is a set of sentences. We are interested in values of $T_\word{th}$ on representable functors (i.e.~single sentences) $\hom_\word{Sen}(-, \phi)$:
\begin{eqnarray*}
T_\word{th}(\hom_\word{Sen}(-, \phi))(\psi) &=& \int_{M\in \mathit{Mod}} \word{th}(M)(\psi)^{\hom(\hom_\word{Sen}(-, \phi), \word{th}(M)(-))}\\
&\approx& \int_{M\in \mathit{Mod}} \word{th}(M)(\psi)^{\word{th}(M)(\phi)}
\end{eqnarray*}
where the isomorphism follows from the Yoneda reduction. Observe that the exponent $\word{th}(M)(\psi)^{\word{th}(M)(\phi)}$ in a $2$-enriched world may be expressed by the implication ``${\word{th}(M)(\phi)} \Rightarrow \word{th}(M)(\psi)$'', or just ``${M \models \phi} \Rightarrow {M \models \psi}$'', where every component of the implication is interpreted as a logical value in the $2$-valued Boolean algebra. Furthermore, ends turn into universal quantifiers, when we move to $2$-enriched world. So, the end $\int_{M \in \mathit{Mod}} \word{th}(M)(\psi)^{\word{th}(M)(\phi)}$ is equivalent to the meta formula ``$\forall_{M\in \word{Mod}} \left({M \models \phi} \Rightarrow {M \models \psi}\right)$'', which is just the definition of logical consequence:
$$ \phi \models_\word{Sen} \psi \;\;\;\mathit{iff}\;\;\; \forall_{M\in \word{Mod}} \left({M \models \phi} \Rightarrow {M \models \psi}\right)$$
The general case, where $\Gamma$ is not necessarily representable, is similar:
$$T_\word{th}(\Gamma)(\psi)  \;\;\;\mathit{iff}\;\;\; \forall_{M\in \word{Mod}} \left(\left(\forall_{\phi \in \Gamma} M \models \phi\right) \Rightarrow {M \models \psi}\right)$$
Therefore, the density product of a satisfaction relation reassembles the semantic consequence relation. 
\end{example}
A density product morphism $T_\tau = \prod_{\tau} \tau$, if exists, is always a part of a monad structure. The unit $\mor{\eta}{\id{C}}{T_\tau}$ is the unique 2-morphism to the product induced by the identity $\mor{\id{\tau}}{\tau}{\tau}$; similarly the multiplication $\mor{\mu}{T_\tau \circ T_\tau}{T_\tau}$ is given as the unique 2-morphism to the product induced by $\epsilon \bullet T_\tau \circ \epsilon$, where $\mor{\epsilon}{T_\tau \circ \tau}{\tau}$ is the product's 2-morphism. By duality, a coproduct morphism $D_\tau = \coprod_{\tau} \tau$, provided it exists, is always a part of a comonad structure. In case of functors between ordinary categories the density coproduct is known as density comonad, and density product is sometimes called a ``codensity monad''. The terminology comes from the fact that a functor $\mor{F}{\catl{A}}{\catl{B}}$ between categories $\catl{A}$ and $\catl{B}$ is dense iff the identity on $\catl{B}$ is the parametrised coproduct of $F$ with itself. In a sense the density comonad on a functor exhibits the ``defect'' of the functor to be dense.
\section{Internal incompleteness theorem}
\label{s:incomplete}
The classical result of Freyd shows that categories that are both small and complete are preorders. Let us recall the argument. If $\catl{C}$ is a small category, then there exists a set of all morphisms of $\catl{C}$ with cardinality $\lambda$. Let us assume that there is a pair of distinct parallel morphisms $\mor{f,g}{A}{B}$ in $\catl{C}$. We may form a product of $\lambda$-copies of $B$, provided $\catl{C}$ is sufficiently complete:
$$\bfig
\node a(0, 1000)[A]
\node b(600, 700)[B]
\node prod(0, 700)[\prod_\lambda B]

\arrow/->/[a`b;f,g]
\arrow/->/[prod`b;\pi_{j \in \lambda}]
\arrow/-->/[a`prod;\exists ! h]
\efig$$
Now, for each index $j \in \lambda$ we may \emph{freely} choose either a morphism $f$ or $g$ to make a cone over $B$'s. There are $\{f, g\}^\lambda$ of such cones. Because, by the property of product $\prod_\lambda B$, each cone uniquely determines a morphism $\mor{h}{A}{\prod_\lambda B}$, the cardinality of the set $\hom(A, \prod_\lambda B)$ is at least $\{f, g\}^\lambda$. This contradicts our claim that the set of all morphism has cardinality $\lambda$, since in ZFC there could be no injection $2^\lambda \rightarrow \lambda$.

The result relies on two fundamental properties of standard set theory. One is non-uniformity of set-indexed collections; or arbitrary richness of set-indexed collections --- for any cardinal $\lambda$ and any set $K$, we may make a free/independent/non-uniform choice of one of the elements of $K$ for each index $j \in \lambda$. Another is the property of being $2$-valued. We say that a set theory is $2$-valued if the set  $2 = 1 \sqcup 1$ forms the subset classifier. By the classical diagonal argument one may show that in any topos with a subobject classifier $\Omega$ there could be no injection $\Omega^A \rightarrow A$. Therefore, the contradiction in the Freyd's argument follows from the fact that the subobject classifier in ZFC has only two elements.

One may wonder if the above properties are crucial to the result of Freyd. And the answer is --- yes, but in two different ways. In late 80's Martin Hyland showed that there exists a small (weakly) complete non-degenerated category internal to the effective topos \cite{pitts} \cite{hyland}. The key argument in his work is that the cones in the mentioned category have to satisfy a suitable smoothness condition (recall Example \ref{e:omega:set}) --- there is no way to form an arbitrary collection $\{f, g\}^\lambda$ as in the above proof. On the other hand, the result of Freyd carries to any cocomplete topos, in particular, to any Grothendieck topos --- no matter how ``big'', or ``complicated'' the subobject classifier in the topos is. In a sense, the second property is used on a higher meta-level than the first one\footnote{The second property refers to the ambient category of the 2-category of internal categories. It is worth pointing out that contrary to some common beliefs the above argument is purely constructive --- even though it may not imply that the set $\hom(A, B)$ has cardinality less than $2$.}, and we shall not investigate it in this paper.

Now, we try to reproduce the result of Freyd in any sufficiently cocomplete 2-category.
\begin{lemma}\label{l:incompleteness}
Let $\catl{W}$ be a 2-category. Consider a pair of parallel morphisms ${\mor{a, b}{X}{C}}$, and a pair of distinct parallel 2-morphisms $\mor{f, g}{a}{b}$ in $\catl{W}$. Let us assume that for a set $\lambda$ the 2-coproduct $\coprod_\lambda X$ exists, and that there is a right Kan extension $\ran_{\nabla}(b \circ \nabla)$ of $\mor{b \circ \nabla}{\coprod_\lambda X}{C}$ along $\mor{\nabla}{\coprod_\lambda X}{X}$, where $\nabla$ is the coproduct codiagonal. Then the set $\hom(a, \ran_\nabla(b \circ \nabla))$ has cardinality at least $2^\lambda$.
\end{lemma}
\begin{proof}
Consider a diagram that satisfies the hypothesis of the lemma:
$$\bfig
\node coprod(800, 1000)[\coprod_\lambda X]
\node x(1400, 1000)[X]
\node c(1800, 1000)[C]

\node xb(800, 600)[X]

\node virtaf(1560, 1065)[]
\node virtbf(1560, 950)[]

\node virtag(1640, 1065)[]
\node virtbg(1640, 950)[]

\node virtseps(1200, 750)[]
\node virtteps(1100, 850)[]

\arrow/{@{>}@/_1em/}/[x`coprod;\iota_{i \in \lambda}]
\arrow|m|/->/[coprod`x;\nabla]

\arrow/->/[coprod`xb;\nabla]
\arrow|r|/{@{>}@/_2em/}/[xb`c;\ran_\nabla(b \circ \nabla)]

\arrow|m|/{@{>}@/^1em/}/[x`c;a]
\arrow|m|/{@{>}@/_1em/}/[x`c;b]

\arrow/=>/[virtaf`virtbf;f]
\arrow|r|/=>/[virtag`virtbg;g]
\arrow/=>/[virtseps`virtteps;\epsilon]
\efig$$
where $\iota_{i \in \lambda}$ are coproduct's injections. We form two cocones --- one by constantly choosing $a$, and another by constantly choosing $b$ for each index $i \in \lambda$. By the universal property of coproduct $\coprod_\lambda X$ these cocones induce unique morphisms $\mor{a \circ \nabla}{\coprod_\lambda X}{C}$ and $\mor{b \circ \nabla}{\coprod_\lambda X}{C}$, respectively. We may form a transformation of cones by independently choosing either a 2-morphism $\mor{f}{a}{b}$ or $\mor{g}{a}{b}$ for each index $i \in \lambda$. There are $\{f, g\}^\lambda$ of such transformations, and by the universal property of 2-coproduct, each transformation uniquely determines a 2-morphism $a \circ \nabla \rightarrow b \circ \nabla$. Therefore, $\hom(a \circ \nabla, b \circ \nabla)$ has cardinality at least $2^\lambda$.
The definition of the right Kan extension $\ran_\nabla(b \circ \nabla)$ says that there is a natural isomorphism:
$$\hom(a, \ran_\nabla(b \circ \nabla)) \approx \hom(a \circ \nabla, b \circ \nabla)$$
thus, by the above, $\hom(a, \ran_\nabla(b \circ \nabla))$ has cardinality at least $2^\lambda$, which completes the proof.
\end{proof}
There is an obvious generalisation of the above lemma, which may be obtained by replacing cardinal $\lambda$ with arbitrary category, and coproduct $\coprod_\lambda X$ with tensor $\lambda \otimes X$. Indeed, by the definition of tensor $\hom(a \circ \pi, b \circ \pi) \approx \hom(\Delta(a), \Delta(b))$, where $\mor{\Delta(a), \Delta(b)}{\lambda}{\hom(X, C)}$ are constant functors assigning everything to $a$ and $b$ respectively, and $\pi$ plays the role of the codiagonal $\nabla$. Therefore $\hom(a, \ran_\pi(b \circ \pi)) \approx \hom(\Delta(a), \Delta(b))$. Choosing discrete $\lambda$ puts no constraints on transformations $\Delta(a) \rightarrow \Delta(b)$ and leads to the conclusion $\hom(\Delta(a), \Delta(b)) \approx \hom(a, b)^\lambda$

Before we state the 2-categorical incompleteness theorem, let us write explicitly definition of a representable poset and of a 2-generating family.
\begin{definition}[Representable poset]\label{d:rep:poset}
An object $A$ from a 2-category $\catl{W}$ is representably posetal if for every object $X \in \catl{W}$, the category $\hom(X, A)$ is a poset.
\end{definition}
\begin{definition}[2-generating family]\label{d:2gen}
A class of objects $G$ from a 2-category $\catl{W}$ is called a 2-generating family if for every pair of 2-morphisms $\alpha, \beta$ between parallel 1-morphisms from an object $A \in \catl{W}$ to an object $B \in \catl{W}$ the following holds: if for every 2-morphism $\tau$ between parallel one morphisms from an object $X \in G$ to object $A$ the equality of compositions holds $\alpha \circ \tau = \beta \circ \tau$, then $\alpha = \beta$.  
\end{definition}
We shall also recall the notion of density in the context of 2-categories.
\begin{definition}[Density]\label{d:density}
A 2-functor $\mor{F}{\catl{C}}{\catl{D}}$ is dense if the 2-functor $A \mapsto \hom_{\catl{D}}(F(-), A)$ is fully faithful. 
\end{definition}
\begin{theorem}[Incompleteness theorem]\label{t:incompleteness}
Let $\catl{W}$ be a locally small 2-category and $G \subset \catl{W}$ a 2-generating family. Furthermore, assume that objects from $G$ have tensors with sets.
If an object $C \in \catl{W}$ has all ad hoc simple products parametrised by $G$, then $C$ is representably posetal.
\end{theorem}
\begin{proof}
Let $X$ be an object in $G \subset \catl{W}$. Let us assume that there exists a pair of distinct 2-morphisms $\mor{f, g}{a}{b} \in \hom(X, C)$, and choose a cardinal $\lambda$ equal to the cardinality of the underlying set of morphisms of $\hom(X, C)$. By Lemma \ref{l:incompleteness}, $\hom(a, \prod_{i \in \lambda}  b)$ has cardinality at least $2^\lambda$, which leads to the contradiction $2^\lambda \leq \lambda$ in ZFC. Therefore, $\hom(X, C)$ is a poset on each $X \in G$, thus by the property of a 2-generating family, $C$ is representably posetal.
\end{proof}
There is also a version of the incompleteness theorem directly using adjunctions to codiagonals (recall Example \ref{e:prod:via:param}).
\begin{corollary}[Special incompleteness theorem]\label{t:special:incompleteness}
Let $A \in \catl{W}$. If for every set $X$ the constant product $\prod_X A$ exists, and the diagonal $\mor{\Delta}{A}{\prod_X A}$ has right adjoint, then $A$ is representably posetal.
\end{corollary}
\begin{example}[Freyd theorem]
The classical Freyd theorem is obtained from Theorem \ref{t:incompleteness} by taking $\catl{W} = \cat$, and recalling that the terminal category $1$ is a 2-generator in $\cat$. Alternatively, one may use the special incompleteness theorem in the following way: in $\cat$  cotensors $X \pitchfork \catl{A} = \catl{A}^X = \prod_X \catl{A}$ exist for any small category $\catl{A}$ and every set $X$; Corollary \ref{t:special:incompleteness} says that if for every $X$ there is a right adjoint to the diagonal $\mor{\Delta}{\catl{A}}{\catl{A}^X}$ then $\catl{A}$ is posetal.   
\end{example}
We shall observe in the next section that for a 2-category of internal categories, the above notion of being representably posetal coincides with the usual notion of an internal poset (Corollary \ref{c:posetal}), and ad hoc products parametrised by discrete objects correspond to the internal products in the usual sense (Corollary \ref{c:bc:condition}).
\begin{definition}[Internal poset]\label{d:internal:poset}
Let $\catl{C}$ be a category with finite limits. A $\catl{C}$-internal poset $A$ is a $\catl{C}$-internal category for which the domain and codomain morphisms $\mor{\mathit{dom},\mathit{cod}}{A_1}{A_0}$ are jointly mono, meaning that the morphism $\mor{\langle\mathit{dom},\mathit{cod}\rangle}{A_1}{A_0 \times A_0}$ is mono.
\end{definition}
Therefore, we may write the following corollary.
\begin{corollary}\label{c:int:freyd}
Let $\cat(\catl{C})$ be the 2-category of categories internal to a finitely complete locally small category $\catl{C}$ that has tensors with sets. If a $\catl{C}$-internal category $C \in \cat(\catl{C})$ has simple ad hoc polymorphism then it is an internal poset.
\end{corollary}
\begin{proof}
The category $\catl{C}$ is a 2-dense subcategory of $\cat(\catl{C})$ spanned on discrete objects (i.e.~the inclusion functor is dense), therefore the class of discrete objects is a 2-generating family.  
\end{proof}
A direct consequence of Corollary \ref{c:int:freyd} is that there are no small complete non-degenerated categories internal to a Grothendieck topos.

We can also get instantly from Theorem \ref{t:incompleteness} the incompleteness theorem for enriched categories.
\begin{corollary}
Let $\catl{V}$ be a monoidal category. If a small $\catl{V}$-enriched category is complete, then it is representably posetal.
\end{corollary}
\begin{proof}
The 2-category of small $\catl{V}$-enriched categories has small coproducts inherited from $\catw{Set}$.
\end{proof}
\begin{example}[$\omega\catw{Set}$ and Hyland's effective topos]
The incompleteness theorem does not work in $\cat(\omega\catw{Set})$ nor in the categories internal to Hyland's effective topos, because these categories do not have ``sufficiently big'' coproducts. Let us show that $\omega\catw{Set}$ does not have even countable coproducts on the terminal object. To obtain a contradiction, assume that a coproduct $\coprod_{n \in N} 1$ exists. Consider the natural number object in $\omega$-sets $N_\omega = \{\tuple{n, n} \colon n \in N\}$. Every $\omega$-function $\mor{k}{1}{N_\omega}$ is uniquely determined by a natural number $k \in N$, and by the universal property of coproduct $\coprod_{n \in N} 1$, every family $n \mapsto k_n$ indexed by natural numbers $n \in N$ uniquely determines an $\omega$-function $\mor{h}{\coprod_{n \in N} 1}{N_\omega}$ with $h(n) = k_n$. Because proofs in $N_\omega$ are disjoint, $h$ is determined by a partially recursive function. This leads to a contradiction since not every function $N \rightarrow N$ is partially recursive.

\end{example}

\section{The associated category}
\label{s:associated}
This section is intended to provide a framework that allows us to better understand 2-categorical models for lambda calculi, and under some conditions embed them into a 2-topos of internal categories. We start with an explicit description of a category associated to an object from a 2-category with a notion of \emph{canonical} discreteness, and then move to a more abstract framework.

In the remaining of the section, we shall use extensively the notion of ``inserter'', which is a particular case of a $\onecat$-weighted limit \cite{kelly}\cite{elements}\cite{borceux}\cite{elephant}.
\begin{definition}[Inserter]\label{d:inserter}
Let us write $2$ for the category $\{0 \to 1 \}$ and consider a functor $\mor{W}{\{\xy \morphism(0,0)|a|/@{>}@<3pt>/<250,0>[\bullet`\bullet;] \morphism(0,0)|b|/@{>}@<-3pt>/<250,0>[\bullet`\bullet;]\endxy\}}{\onecat}$ that maps a category $\{\xy \morphism(0,0)|a|/@{>}@<3pt>/<250,0>[\bullet`\bullet;] \morphism(0,0)|b|/@{>}@<-3pt>/<250,0>[\bullet`\bullet;]\endxy\}$ to the diagram $\{\xy \morphism(0,0)|a|/@{>}@<3pt>/<250,0>[1`2;0] \morphism(0,0)|b|/@{>}@<-3pt>/<250,0>[1`2;1]\endxy\}$ in $\onecat$, where functors $0$ and $1$ are constant and map the whole category to $0$ and $1$ respectively. Let $\catl{W}$ be a 2-category, and $F$ a functor $\{\xy \morphism(0,0)|a|/@{>}@<3pt>/<250,0>[\bullet`\bullet;] \morphism(0,0)|b|/@{>}@<-3pt>/<250,0>[\bullet`\bullet;]\endxy\} \rightarrow \catl{W}$. An inserter of $F$ is a representation of:
$$X \mapsto \hom(W(-), \hom_{\catl{W}}(X, F(-)))$$
That is, an object $I \in \catl{W}$ and an isomorphism between categories:
$$\hom(W(-), \hom_{\catl{W}}(X, F(-))) \approx \hom_{\catl{W}}(X, I)$$
natural in $X$.
\end{definition}
Let us rewrite the definition of an inserter in more explicit terms. A functor $\mor{F}{\{\xy \morphism(0,0)|a|/@{>}@<3pt>/<250,0>[\bullet`\bullet;] \morphism(0,0)|b|/@{>}@<-3pt>/<250,0>[\bullet`\bullet;]\endxy\}}{\catl{W}}$ corresponds to a diagram $\{\xy \morphism(0,0)|a|/@{>}@<3pt>/<250,0>[A`B;f] \morphism(0,0)|b|/@{>}@<-3pt>/<250,0>[A`B;g]\endxy\}$ in $\catl{W}$. A natural transformation in $\hom(W(-), \hom(X, F(-)))$ chooses a morphism $\mor{x}{X}{A}$ together with a 2-morphism $\mor{\alpha}{f \circ x}{g \circ x}$ like on the picture:
$$\bfig
\node a1(400, 1000)[X]
\node daa(800, 1000)[A]
\node da(1200, 1000)[B]
\node f(1600, 1000)[f \circ x]
\node g(2000, 1000)[g \circ x]
\arrow/->/[a1`daa;x]
\arrow|r|/{@{>}@/_1em/}/[daa`da;g]
\arrow/{@{>}@/^1em/}/[daa`da;f]
\arrow/=>/[f`g;\alpha]
\efig$$
We call a pair $\langle\mor{x}{X}{A}, \mor{\alpha}{f \circ x}{g \circ x} \rangle$ an ``inserter cone over $X$''. A morphism between parallel natural transformations $W(-) \rightarrow \hom_\catl{W}(X, F(-))$ is a \emph{modification}. If $\langle\mor{x}{X}{A}, \mor{\alpha}{f \circ x}{g \circ x} \rangle$ and $\langle\mor{x'}{X}{A}, \mor{\alpha'}{f \circ x'}{g \circ x'} \rangle$ are two inserter cones over $X$ induced by natural transformations $W(-) \rightarrow \hom_\catl{W}(X, F(-))$, then a modification between the natural transformations corresponds to a single $2$-morphism $\mor{\gamma}{x}{x'}$ in $\catl{W}$ such that $(g \circ \gamma) \bullet \alpha = \alpha' \bullet (f \circ \gamma)$. Therefore, we may write $\mathit{Inserter}(X; f, g)$ for the category of inserter cones over X of the shape of $F$, which is isomorphic to $\hom(W(-), \hom_\catl{W}(X, F(-)))$. Then, the assignment $X \mapsto \mathit{Inserter}(X; f, g)$ extends by composition to a functor:
$$\mor{\mathit{Inserter}(-; f, g)}{\catl{W}^{op}}{\onecat}$$
The inserter $I$ of $f, g$ is a 2-representation
$$\mor{\hom_\catl{W}(-, I)}{\catl{W}^{op}}{\onecat}$$
of $\mathit{Inserter}(-; f, g)$. That is, the inserter is an object $I$ together with a morphism $\mor{i}{I}{A}$ and a 2-morphism $\mor{\pi}{f \circ i}{g \circ i}$ that is universal in the following sense:
$$\bfig

\node a1(200, 1000)[I]
\node daa(800, 1000)[A]
\node da(1400, 1000)[B]
\node f(1800, 1200)[f \circ i]
\node g(2200, 1200)[g \circ i]

\node f2(1800, 1000)[f \circ x]
\node g2(2200, 1000)[g \circ x]

\node a0(400, 1300)[X]

\arrow/->/[a1`daa;i]
\arrow|r|/{@{>}@/_1em/}/[daa`da;g]
\arrow/{@{>}@/^1em/}/[daa`da;f]
\arrow/=>/[f`g;\pi]
\arrow/=>/[f2`g2;\alpha]

\arrow/->/[a0`daa;x]
\arrow/-->/[a0`a1;h_x]
\efig$$
for every diagram $\langle\mor{x}{X}{A}, \mor{\alpha}{f \circ x}{g \circ x} \rangle$ there exists a unique morphism $\mor{h_x}{X}{I}$ such that $i \circ h_x = x$ and $\pi \circ h_x = \alpha$; and for every diagram $\langle\mor{x'}{X}{A}, \mor{\alpha'}{f \circ x'}{g \circ x'} \rangle$ and a 2-morphism $\mor{\gamma}{x}{x'}$ that is a morphism of diagrams, i.e.~$(g \circ \gamma) \bullet \alpha = \alpha' \bullet (f \circ \gamma)$, there exists a unique 2-morphism $\mor{h_{\gamma}}{h_x}{h_{x'}}$ such that $i \circ  h_{\gamma} = \gamma$. 

By the above characterisation, we instantly get the following corollary.
\begin{corollary}\label{c:dicrete:inserter}
In any 2-category a morphism $\mor{i}{I}{A}$ of an inserter $\tuple{I, \mor{i}{I}{A}}$ is discrete --- i.e.~it is representably faithful and conservative, which means that for every object $X$ the functor $\hom(X, i)$ is faithful and conservative.
\end{corollary}
If $\disc(\catl{W})$ gives the canonical notion of discreteness on a finitely complete 2-category $\catl{W}$, then with every object $A \in \catl{W}$ we may associate a $\disc(\catl{W})$-internal category $\catl{A}$. 
Given $A \in \catl{W}$ we define the ``object of objects'' $\catl{A}_0$ as $|A|$. Then we shall define the ``object of morphisms'' $\catl{A}_1$ as the inserter of the following diagram (notice that $|A \times A| \approx |A| \times |A|$ since $|{-}|$ is right adjoint): 
$$\bfig
\node a1(0, 1000)[\catl{A}_1]
\node daa(800, 1000)[|A \times A|]
\node da(1400, 1000)[|A|]
\node a(1800, 1000)[A]
\arrow/->/[a1`daa;\langle\mathit{dom}, \mathit{cod}\rangle]
\arrow|r|/{@{>}@/_1em/}/[daa`da;\pi_{1_{|A|}}]
\arrow/{@{>}@/^1em/}/[daa`da;\pi_{0_{|A|}}]
\arrow/->/[da`a;\epsilon_A]
\efig$$
together with the ``choosing'' 2-morphism:
${\mor{\alpha}{\epsilon_A \circ \mathit{dom}}{\epsilon_A \circ \mathit{cod}}}$.
We have to show that $\catl{A}_1$ is discrete. However this is a straightforward consequence of Corollary \ref{c:dicrete:inserter}.
\begin{corollary}
An inserter  $\tuple{I, \mor{i}{I}{A}}$  on a discrete object $A$ is a discrete object.
\end{corollary}
The internal identity ${\mor{\eta_\catl{A}}{\catl{A}_0}{\catl{A}_1}}$ is given 
as the unique morphism to the inserter induced by the identity 2-morphism on:
$$\bfig
\node a1(0, 1000)[\catl{A}_1]
\node daa(800, 1000)[|A \times A|]
\node da(1400, 1000)[|A|]
\node a(1800, 1000)[A]
\node a0(400, 1350)[\catl{A}_0]

\arrow/->/[a1`daa;\langle\mathit{dom}, \mathit{cod}\rangle]
\arrow|r|/{@{>}@/_1em/}/[daa`da;\pi_{1_{|A|}}]
\arrow/{@{>}@/^1em/}/[daa`da;\pi_{0_{|A|}}]
\arrow/->/[da`a;\epsilon_A]
\arrow/->/[a0`daa;|\Delta_A|]
\arrow/-->/[a0`a1;\eta_{\catl{A}}]
\efig$$
To define the internal composition, let us first form the pullback:
$$\bfig
\node a0(1000, 1000)[\catl{A}_0]
\node da1(1000, 1250)[\catl{A}_1]
\node ca1(600, 1000)[\catl{A}_1]
\node a2(600, 1250)[\catl{A}_2]
\arrow|r|/->/[da1`a0;\mathit{dom}]
\arrow|r|/->/[ca1`a0;\mathit{cod}]
\arrow/->/[a2`da1;p_1]
\arrow/->/[a2`ca1;p_2]
\efig$$
and take the composition $\mor{\mu_{\catl{A}}}{\catl{A}_2}{\catl{A}_1}$ to be the unique morphism to the inserter induced by the 2-morphism:
$$\mor{\alpha p_2 \bullet \alpha p_1}{\epsilon_A \circ \mathit{dom} \circ p_1}{\epsilon_A \circ \mathit{cod} \circ p_2}$$
of the diagram:
$$\bfig
\node a1(0, 1000)[\catl{A}_1]
\node daa(800, 1000)[|A \times A|]
\node da(1400, 1000)[|A|]
\node a(1800, 1000)[A]
\node a2(300, 1350)[\catl{A}_2]

\arrow/->/[a1`daa;\langle\mathit{dom}, \mathit{cod}\rangle]
\arrow|r|/{@{>}@/_1em/}/[daa`da;\pi_{1_{|A|}}]
\arrow/{@{>}@/^1em/}/[daa`da;\pi_{0_{|A|}}]
\arrow/->/[da`a;\epsilon_A]
\arrow/->/[a2`daa;\langle\mathit{dom} \circ p_1, \mathit{cod} \circ p_2\rangle]
\arrow/-->/[a2`a1;\mu_{\catl{A}}]
\efig$$
\begin{definition}[Canonically associated category]\label{d:can:associated:category}
Let $\catl{W}$ be a finitely complete 2-category with a canonical notion of discreteness. With the notation as above, we define an associated category to an object $A \in \catl{W}$ to be the $\mathit{Disc}(\catl{W})$-internal category $\catl{A} = \langle \catl{A}_0, \catl{A}_1, 
\xy \morphism(0,0)|a|/@{>}@<3pt>/<400,0>[\catl{A}_1`\catl{A}_0;\mathit{dom}] \morphism(0,0)|b|/@{>}@<-3pt>/<400,0>[\catl{A}_1`\catl{A}_0;\mathit{cod}]\endxy 
,\catl{A}_0 \to^{\eta_\catl{A}} \catl{A}_1, \catl{A}_2 \to^{\mu_{\catl{A}}} \catl{A}_1 \rangle$.
\end{definition}
Similarly, every morphism $\mor{f}{A}{B}$ induces an internal functor $\mor{F}{\catl{A}}{\catl{B}}$, and every 2-morphism $\mor{\tau}{f}{g}$ induces an internal natural transformation $\mor{\tau}{F}{G}$ between internal functors induced by $f$ and $g$. This gives a 2-functor $\mor{E}{\catl{W}}{\cat(\mathit{Disc}(\catl{W}))}$. We shall see that
$\disc(\catl{W})$ is a 2-dense subcategory of $\catl{W}$ iff ${\mor{E}{\catl{W}}{\cat(\disc(\catl{W}))}}$ is a fully faithful embedding. One then instantly gets the following: the 2-functor $\mor{E}{\cat(\catl{C})}{\cat(\disc(\cat(\catl{C})))}$ is a 2-equivalence of 2-categories for any category $\catl{C}$ with pullbacks.
\begin{corollary}\label{c:posetal}
A category $\catl{A}$ internal to a finitely complete category $\catl{C}$ is representably posetal iff it is an internal poset.
\end{corollary}
\begin{proof}%
%
Let $X$ be an object in $\catl{C}$. We shall think of $X$ as a discrete $\catl{C}$-internal category. An internal functor $\mor{f}{X}{\catl{A}}$ is tantamount to a single morphism $\mor{f}{X}{\catl{A}_0}$ in $\catl{C}$. An internal natural transformation between such functors $\mor{f, g}{X}{\catl{A}}$ consists of a morphism $\mor{\tau}{X}{\catl{A}_1}$ satisfying $\tuple{\mathit{dom} \circ \tau, \mathit{cod} \circ \tau} = \tuple{f, g}$. Therefore $\tuple{\mathit{dom}, \mathit{cod}}$ is mono precisely when over any $\tuple{f, g}$ there is at most one internal natural transformation. On the other hand if $\tuple{\mathit{dom}, \mathit{cod}}$ is mono, the condition $\tuple{\mathit{dom} \circ \tau, \mathit{cod} \circ \tau} = \tuple{f_0, g_0}$ ensures that $\hom(\catl{X}, \catl{A})$ is posetal for any $\catl{C}$-internal category $\catl{X}$.
\end{proof}
There is also a construction in the other direction $\mor{I}{\cat(\mathit{Disc}(\catl{W}))}{\catl{W}}$, provided $\catl{W}$ has enough (weighted) colimits. But first, let us recall the definition of a family fibration from Chapter 7.3 of \cite{jacobs}.
\begin{definition}[Externalisation of a category]\label{d:externalisation}
For every category $\catl{A}$ internal to a finitely complete category $\catl{C}$ one may construct a split indexed category (the externalisation of a category):
${\mor{\mathit{fam}(\catl{A})}{\catl{C}^{op}}{\onecat}}$
as follows:
\begin{itemize}
\item $\mathit{fam}(\catl{A})(X)$ is the category whose objects are morphisms $X \rightarrow \catl{A}_0$ in $\catl{C}$, whose morphisms from an object $\mor{x}{X}{\catl{A}_0}$ to an object $\mor{y}{X}{\catl{A}_0}$ are morphisms $\mor{f}{X}{\catl{A}_1}$ in $\catl{C}$ such that $\langle \mathit{dom}, \mathit{cod} \rangle \circ f = \langle x, y \rangle$ and with the identities and compositions inherited from $\catl{A}$
\item for a morphism $\mor{f}{X}{Y}$ the functor $\mathit{fam}(\catl{A})(f) = (-) \circ f$ is the post-composition with $f$.
\end{itemize}
\end{definition}
Let $\mor{F}{\disc(\catl{W})}{\catl{W}}$ be the inclusion from the category of discrete objects. Consider a $\disc(\catl{W})$-internal category $\catl{A}$ together with its externalisation\newline 
${\mor{\mathit{fam}(\catl{A})}{\disc(\catl{W})^{op}}{\onecat}}$.
The corresponding object $I(\catl{A}) \in \catl{W}$, if it exists, is the colimit of $F$ weighted by $\mathit{fam}(\catl{A})$. Therefore, if $\catl{W}$ has enough (weighted) colimits then there exists a 2-functor $\mor{I}{\cat(\disc(\catl{W}))}{\catl{W}}$, which is left adjoint to $\mor{E}{\catl{W}}{\cat(\mathit{Disc}(\catl{W}))}$. 
%

Instead of directly proving the above facts, we generalise the construction of an associated category to any notion of discreteness and prove more general theorems. Let us first generalise the construction of family fibration from Definition \ref{d:externalisation}. Because fibrations are equivalent to indexed categories, we use these concepts interchangeably.
\begin{definition}[Generalised family fibration]\label{d:family:fibration}
Let $\mor{F}{\catl{C}}{\catl{W}}$ be a functor from a 1-category to a 2-category. Every object $A \in \catl{W}$ induces a split indexed category:
$\mor{\hom(F(-), A)}{\catl{C}^{op}}{\onecat}$,
which we shall call ``family fibration'' and denote by $\mathit{fam}_F(A)$.  
\end{definition}
\begin{example}[Canonical family fibration]\label{e:canonical:family}
Let $\catl{A}$ be a $\catl{C}$-internal category. Its externalisation
${\mor{\mathit{fam}(\catl{A})}{\catl{C}^{op}}{\onecat}}$
coincides with the family fibration in the sense of Definition~\ref{d:externalisation}:
$\mor{\mathit{fam}_F(\catl{A})}{\disc(\cat(\catl{C}))^{op}}{\onecat}$ 
where ${\disc(\cat(\catl{C})) \approx \catl{C}}$
and $\mor{F}{\catl{C}}{\onecat(\catl{C})}$ gives the canonical notion of discreteness. More generally, if $\catl{A}$ is a category relative to a monoidal fibration \cite{gag} \cite{mrp} \cite{mike}, then its externalisation as defined in Chapter 1.5 of \cite{mrp} also coincides with the family fibration.
\end{example}
%
The assignment $A \mapsto \mathit{fam}_F(A)$
extends to a 2-functor: $\mor{\mathit{fam}_F}{\catl{W}}{\onecat^{\catl{C}^{op}}}$ which will be called ``the family functor''.
We shall recall the definitions of a generic object, locally small, and small indexed category \cite{borceux} \cite{elephant} \cite{jacobs} \cite{phoa}.
\begin{definition}[Generic object]\label{d:generic:object}
A split indexed category ${\mor{\Theta}{\catl{C}^{op}}{\onecat}}$ has a generic object $\Omega \in \catl{C}$ if its underlying discrete indexed category:
$${\catl{C}^{op} \to^{\Theta} \onecat \to^{|{-}|} \catw{Set}}$$
is represented by:
$$\hom_{\catl{C}}(-, \Omega)$$
\end{definition}
\begin{definition}[Locall smallness]\label{d:locally:small}
A split indexed category $\mor{\Theta}{\catl{C}^{op}}{\onecat}$ is locally small if for every object $I \in \catl{C}$ and every pair of objects $X, Y \in \Theta(I)$ there exists an object $\underline{\hom(X, Y)} \in \catl{C}$ together with a morphism $\mor{p}{\underline{\hom(X, Y)}}{I}$, and a vertical morphism $\mor{\chi}{\Theta(p)(X)}{\Theta(p)(Y)}$ over $\underline{\hom(X, Y)}$ such that for any morphism $\mor{q}{J}{I} \in \catl{C}$ and any vertical morphism $\mor{\beta}{\Theta(q)(X)}{\Theta(q)(Y)}$ over $J$ there exists a unique morphism $\mor{h}{J}{\underline{\hom(X, Y)}}$ such that $p \circ h = q$ and $\Theta(h)(\chi) = \beta$.
\end{definition}
\begin{definition}[Smallness]\label{d:small}
A split indexed category is small if it has a generic object and is locally small.
\end{definition}
It is well-known that (split) small categories indexed over a category $\catl{C}$ with finite limits are equivalent to $\catl{C}$-internal categories (Proposition 7.3.8 in \cite{jacobs}). We show that if $\catl{C}$ is a coreflective subcategory of a finitely complete 2-category $\catl{W}$, then $\catl{C}$-indexed family fibrations of $\catl{W}$ are small, thus have associated $\catl{C}$-internal categories.
\begin{theorem}\label{t:famgeneric}
An indexed category $\mathit{fam}_F(A)$ has a generic object iff $F$ has a \mbox{(1-)right} adjoint. 
\end{theorem}
\begin{proof}
The theorem is almost tautological. The definition of an adjunction says that for every $A$ there is a natural isomorphism:
$\hom(F(-), A) \approx \hom(-, U(A))$
between $\catw{Set}$-valued functors, but this is exactly the definition of a generic object $\Omega = U(A)$.
\end{proof}
\begin{theorem}
If $\catl{W}$ has (weighted) finite limits and an adjunction $\xy \morphism(0,0)|a|/@{>}@<3pt>/<400,0>[\catl{W}`\catl{C};U] \morphism(0,0)|b|/@{<-}@<-3pt>/<400,0>[\catl{W}`\catl{C};F]\endxy$ makes $\catl{C}$ a coreflective subcategory of $\catl{W}$, then for every object $A \in \catl{W}$ family fibration $\mathit{fam}_F(A)$ is locally small.
\end{theorem}
\begin{proof}
Let $I$ be an object in $\catl{C}$, and $\mor{x, y}{F(I)}{A}$ two parallel morphisms. Let us write $\mor{\pi}{\underline{\hom(x, y)}}{F(I)}$ for the inserter of $x, y$, and $\mor{\alpha}{x \circ \pi}{y \circ \pi}$ for the inserter's 2-morphism. We shall show that such data mapped by the functor $U$ make $\mathit{fam}_F(A)$ a locally small fibration. Formally, let $p = \eta^{-1}_I \circ U(\pi)$, and $\chi = \alpha \circ \epsilon_{\underline{\hom(x, y)}}$. Observe that $\chi$ is really a 2-morphism $x \circ F(p) \rightarrow y \circ F(p)$:
$$\bfig
\node fuh(0, 1000)[FU(\underline{\hom(x, y)})]
\node fufi(800, 1000)[FUF(I)]
\node fi(1500, 1000)[F(I)]

\node h(0, 1400)[\underline{\hom(x, y)}]
\node tfi(800, 1400)[F(I)]

\arrow/->/[fuh`fufi;FU(\pi)]
\arrow/->/[fufi`fi;F(\eta^{-1}_I)]
\arrow/->/[fufi`tfi;\epsilon_{F(I)}]
\arrow/->/[fuh`h;\epsilon_{\underline{\hom(x, y)}}]
\arrow/->/[h`tfi;\pi]
\arrow/->/[fi`tfi;\mathit{id}]
\efig$$
The square commutes by naturality of the counit $\epsilon$, and commutativity of the triangle on the right side follows from triangle equality of the adjunction. We have to show that for any $\mor{q}{J}{I}$ and any 2-morphism $\mor{\beta}{x \circ F(q)}{y \circ F(q)}$ there exists a unique morphism $\mor{h}{J}{U(\underline{\hom(x, y)})}$ such that $p \circ h = q$ and $\chi \circ F(h) = \beta$. By the definition of inserter $\underline{\hom(x, y)}$, we get a morphism $\mor{\widehat{h}}{F(J)}{\underline{\hom(x, y)}}$ like on the diagram
$$\bfig
\node a1(200, 1000)[\underline{\hom(x, y)}]
\node daa(800, 1000)[F(I)]
\node da(1400, 1000)[A]
\node f(475, 700)[x \circ \pi]
\node g(930, 700)[y \circ \pi]

\node f2(400, 500)[x \circ F(q)]
\node g2(1000, 500)[y \circ F(q)]

\node a0(400, 1300)[F(J)]

\arrow/->/[a1`daa;\pi]
\arrow|r|/{@{>}@/_1em/}/[daa`da;y]
\arrow/{@{>}@/^1em/}/[daa`da;x]
\arrow/=>/[f`g;\alpha]
\arrow/=>/[f2`g2;\beta]

\arrow/->/[a0`daa;F(q)]
\arrow/-->/[a0`a1;\widehat{h}]
\efig$$
which via transposition gives a morphism $\mor{h}{J}{U(\underline{\hom(x, y)})}$. The above conditions follows directly from the coreflectivity of $\catl{C}$ and the definition of the inserter. We have $p \circ h = \eta^{-1}_I \circ U(\pi) \circ U(\widehat{h}) \circ \eta_J = \eta^{-1}_I \circ UF(q) \circ \eta_J = q$, and $\chi \circ F(h) = \alpha \circ \epsilon_{\underline{\hom(x, y)}} \circ F(h) = \alpha \circ \widehat{h} = \beta$. For the uniqueness, let us assume that $\mor{h}{J}{U(\underline{\hom(x, y)})}$ is such that $p \circ h = q$ and $\chi \circ F(h) = \beta$. Since $h$ and $\widehat{h}$ uniquely determines each other, it suffices to show the following $\pi \circ \widehat{h} = \pi \circ \epsilon_{\underline{\hom(x, y)}} \circ F(h) = F(p) \circ F(h) = F(q)$, and $\alpha \circ \widehat{h} = \alpha \circ \epsilon_{\underline{\hom(x, y)}} \circ F(h) = \chi \circ F(h) = \beta$.
\end{proof}
\begin{corollary}\label{t:famsmall}
If $\catl{W}$ has (weighted) finite limits and the adjunction $\xy \morphism(0,0)|a|/@{>}@<3pt>/<400,0>[\catl{W}`\catl{C};U] \morphism(0,0)|b|/@{<-}@<-3pt>/<400,0>[\catl{W}`\catl{C};F]\endxy$ makes $\catl{C}$ a coreflective subcategory of $\catl{W}$, then every indexed category $\mathit{fam}_F(A)$ is small.
\end{corollary}
\begin{theorem}[Representation theorem]\label{t:representation}
Let $\catl{W}$ be a 2-category with a notion of discreteness and having finite (weighted) limits.
With every object $A \in \catl{W}$ we may associate, in a canonical way, a $\disc(\catl{W})$-internal category.
Moreover, this assignment makes $\catl{W}$ a full (necessarily dense) 2-subcategory of $\cat(\disc(\catl{W}))$ iff $\disc(\catl{W})$ is a dense subcategory of $\catl{W}$. 
\end{theorem}
\begin{proof} Density of $\disc(\catl{W})$ in $\catl{W}$ by definition is equivalent to saying that the 2-functor $\mor{\mathit{fam}_F}{\catl{W}}{\onecat^{\disc(\catl{W})^{op}}}$
is fully faithful. It is then also essentially surjective on objects. Therefore, by Corollary \ref{t:famsmall}, $\catl{W}$ is equivalent to a full subcategory of $\disc(\catl{W})$-internal categories.
\end{proof}
\begin{example}[$\onecat$ with canonical discreteness]
The canonical externalisation of a category $\catl{C}$ gives the usual family fibration $\mor{\mathit{fam}(\catl{C})}{\catw{Set}^{op}}{\onecat}$. This fibration is small precisely when category $\catl{C}$ is small. The category associated to $\catl{C}$ is (equivalent to) the same category.   
\end{example}
\begin{example}[$\onecat$ with $0$]
The subcategory of $\onecat$ consisting of a single empty category $0$ gives a non-dense notion of discreteness on $\onecat$.
Since $\catl{C}^0 \approx 1$ for any category $\catl{C}$, there is only one associated category to every object in $\onecat$.
\end{example}
\begin{example}[$\onecat$ with $1$]
The subcategory of $\onecat$ consisting of a terminal category $1$ does not give a notion of discreteness on $\onecat$, simply because the terminal category functor $1 \rightarrow \onecat$ does not have right adjoint. However, the terminal category is a 2-generator in $\onecat$. The family fibration does not loose any information about objects in $\onecat$, but every non-trivial fibration is not small, therefore does not have the associated category.
\end{example}
We shall write
${\mor{E}{\catl{W}}{\cat(\disc(\catl{W}))}}$
for the functor from Theorem \ref{t:representation} representing an object from $\catl{W}$ as an internal category.
\begin{lemma}\label{l:preservation}
Let $\catl{W}$ be a 2-category with a notion of discreteness. The functor
$\mor{E}{\catl{W}}{\cat(\disc(\catl{W}))}$
preserves limits and discrete objects.  
\end{lemma}
\begin{proof}
It preserves limits by 2-Yoneda lemma, and discrete objects by the definition of discreteness.
\end{proof}
\begin{theorem}\label{t:external:models}
Let $\catl{W}$ be a finitely (weighted) complete 2-category with a notion of discreteness. If an object $A \in \catl{W}$ has internal connectives (internal terminal/initial value, internal (closed) products, coproducts) then its associated category $E(A)$ has corresponding connectives in the usual sense. Moreover, if discrete objects are dense, then the converse holds as well.
\end{theorem}
\begin{proof}
One direction follows from Lemma \ref{l:preservation} and the fact that 2-functors preserve adjunctions. The other direction follows from the same facts plus Theorem \ref{t:representation} saying that $\catl{W}$ is a full subcategory of $\cat(\disc(\catl{W}))$ provided $\disc(\catl{W})$ is dense.
\end{proof}
\begin{corollary}
Theorem \ref{d:parametrised:lambda} from Section \ref{s:models} holds.
\end{corollary}
The notion of an associated category allows us to better understand the Beck-Chevalley condition for fibred (co)products. Let us recall that a fibration represented as an indexed category $\mor{\Phi}{\catl{C}^{op}}{\onecat}$ over a finitely complete category $\catl{C}$ has (co)products if for each morphism $\mor{s}{X}{Y}$ in $\catl{C}$ the functor $\Phi(s)$ has right $\prod_s$ (resp. left $\coprod_s$) adjoint. Furthermore, the (co)products satisfy the Beck-Chevalley condition if for every pullback:
$$\bfig
\node y(1000, 1000)[Y]
\node x(1000, 1250)[X]
\node i(600, 1000)[I]
\node p(600, 1250)[P]
\arrow|r|/->/[x`y;s]
\arrow|r|/->/[i`y;i]
\arrow/->/[p`x;\pi_1]
\arrow/->/[p`i;\pi_2]
\efig$$
the canonical natural transformation $\Phi(i) \circ \prod_s \rightarrow \prod_{\pi_2} \circ \pi_1$ (resp. $\coprod_{\pi_2} \circ \pi_1 \rightarrow \Phi(i) \circ \coprod_s$) is an isomorphism. 
\begin{corollary}\label{c:bc:condition}
Let $\catl{W}$ be a 2-category with a notion of discreteness. An object $A \in \catl{W}$ has polymorphic (co)products iff its family fibration has (co)products along all morphisms. Moreover if $\catl{W}$ has finite limits, then these (co)products are stable iff in the family fibration (co)products satisfy the Beck-Chevalley condition.
\end{corollary}
The above corollary together with Theorem \ref{t:external:models} imply that our 2-categorical models for polymorphism externalise to fibrational models \cite{jacobs} \cite{plcat}. However, if discrete objects are not dense, we may not rely on the external fibrational semantics. To see this, consider a monoidal-enriched category --- its externalisation gives the usual family fibration on the \emph{underlying} category; therefore fibrational semantics discard enrichment and collapse to semantics in an ordinary category. 
%

We close this section by merely mentioning the left adjoint to the representation functor.
\begin{theorem}
Let $\catl{W}$ be a 2-category and assume that there is an adjunction $\xy \morphism(0,0)|a|/@{>}@<3pt>/<400,0>[\catl{W}`\catl{C};U] \morphism(0,0)|b|/@{<-}@<-3pt>/<400,0>[\catl{W}`\catl{C};F]\endxy$ making $\catl{C}$ a coreflective subcategory of $\catl{W}$.
The 2-functor $\mor{\mathit{fam}_F(-)}{\catl{W}}{\onecat^{\catl{C}^{op}}}$ has left adjoint ${\mor{L}{\onecat^{\catl{C}^{op}}}{\catl{W}}}$ expressed as the coend:
$$L(H) = \int^{X \in \catl{C}} H(X) \times F(X)$$
provided $\catl{W}$ is sufficiently cocomplete. Moreover, if $\catl{W}$ is finitely complete, the above formula induces adjunction $\catl{W} \leftrightarrows\cat(\catl{C})$.
\end{theorem}
\begin{proof}
Let $\mor{H}{\catl{C}^{op}}{\onecat}$ be an indexed category, and $A$ an object in $\catl{W}$. There are natural 2-isomorphisms:
\begin{center}
\begin{tabular}{c}
$\hom(\int^{X \in \catl{C}} H(X) \times F(X), A)$\\
\hline\hline
$\int_{X \in \catl{C}} \hom(H(X) \times F(X), A)$\\
\hline\hline
$\int_{X \in \catl{C}} \hom(H(X), \hom(F(X), A))$\\
\hline\hline
$\hom(H, \hom(F(-), A))$\\
\hline\hline
$\hom(H, \mathit{fam}_F(A))$\\
\end{tabular}
\end{center}
where the first isomorphism exists because $\hom$-functors turn colimits into limits, the second is the definition of the tensor with a category, the third is the definition of the object of natural transformation, and the last one is the definition of the family fibration. By Theorem \ref{t:representation} the above restricts to the adjunction  $\catl{W} \leftrightarrows\cat(\catl{C})$.
\end{proof} 



\section{Conclusions}
\label{s:conclusions}
In the paper we showed that a natural categorical framework for lambda calculi is encapsulated by a 2-category with a notion of discreteness: we provided a robust concept of internal closedness and a concept of polymorphism generalising notions of cartesian closedness and products for fibrations and for internal categories. We characterised ``ad hoc'' polymorphism and proved a 2-categorical version of Freyd incompleteness theorem: arbitrary ``ad hoc'' polymorphism is not possible in non-degenerate objects. As a simple corollary we obtained the Freyd theorem for categories internal to any \emph{tensored} category --- that is, if a locally small category $\catl{C}$ has ``constant coproducts'' $\coprod_X A$ for any small set $X$ and each object $A \in \catl{C}$, then a category \emph{internal} to $\catl{C}$ that has $\catl{C}$-indexed (internal) products is necessary \emph{internally} posetal. Finally, we developed the theory of associated categories linking our 2-categorical models with well-studied fibrational/internal models for lambda calculi. We generalised the concept of externalisation of a relative category to the concept of externalisation of an object in an arbitrary 2-category with a notion of discreteness. We showed that the process of externalisation preserves models for lambda calculi, and proved, that if a 2-category is finitely weighted complete, then the fibration obtained from externalisation of an object is \emph{small}, thus equivalent to an internal category. 
%


\begin{thebibliography}{56}
%
\bibitem{lambek} J. Lambek and P. J. Scott, \textit{Introduction to Higher-Order Categorical Logic}, Cambridge Studies in Advanced Mathematics, 1988.

\bibitem{jonpart}
P. T. Johnstone, \textit{Fibrations and partial products in a 2-category},      Applied Categorical Structures, Volume 1, Issue 2 (1993)

\bibitem{proarrows}
R. J. Wood, \textit{Abstract pro arrows I}, Cahiers de Topologie et G\'{e}om\'{e}trie Diff\'{e}rentielle Cat\'{e}goriques 23.3 (1982)

\bibitem{proarrows2}
R. J. Wood, \textit{Abstract pro arrows II}, Cahiers de Topologie et G\'{e}om\'{e}trie Diff\'{e}rentielle Cat\'{e}goriques 26.2 (1982)

\bibitem{yoneda}
R. Street, R. Walters, \textit{Yoneda structures on 2-categories}, Journal of Pure and Applied Algebra 50 (1978)

\bibitem{2topos}
M. Weber, \textit{Yoneda Structures from 2-toposes}, Applied Categorical Structures 15-3 (2007)

\bibitem{barr}
M. Barr and C. Wells, \textit{Toposes, Triples and Theories}, Version 1.1 (2002)

\bibitem{reynolds} J. Reynolds, \textit{Polymorphism is not Set-Theoretic}, Semantics of Data Types, Lecture Notes in Computer Science 173 (1984)

\bibitem{jacobspaper}
B. Jacobs \textit{Comprehension categories and the semantics of type dependency}, Journal Theoretical Computer Science, Volume 107 Issue 2, January 18 (1993)

\bibitem{hop}
B. Day, R. Street \textit{Monoidal bicategories and Hopf algebroids}, Advances in Mathematics 129 (1997) 

\bibitem{isbellspace}
J. Isbell, \textit{General function spaces, products and continuous lattices}, Mathematical Proceedings of the Cambridge Philosophical Society, Volume 100, Issue 02 (1986)
\bibitem{ccspace}
M. Escardó, J. Lawson, A. Simpson, \textit{Comparing Cartesian closed categories of (core) compactly generated spaces}, Topology and its Applications 143 (2004)

\bibitem{top}
N. E. Steenrod, \textit{A convenient category of topological spaces}, The Michigan Mathematical Journal 14 (1967)

\bibitem{rsys} J. C. Reynolds, \textit{Types, abstraction and parametric polymorphism.},  Information Processing 83 (1983)

\bibitem{plotkin}
M. Abadi, G.D. Plotkin, \textit{A Per Model of Polymorphism and Recursive Types.}, Logic in Computer Science (1990)


\bibitem{pitts} A. Pitts, \textit{Polymorphism is Set Theoretic, Constructively}, Category Theory and Computer Science, Lecture Notes in Computer Science 283 (1987)


\bibitem{hyland}
J.M.E. Hyland, \textit{A Small Complete Category}, Annals of Pure and Applied Logic, 40-2 (1988)

\bibitem{kelly} G. M. Kelly, \textit{Basic Concepts of Enriched Category Theory}, London Mathematical Society Lecture Note Series No.64 (1982)

\bibitem{elements} G. M. Kelly, R. Street, \textit{Review of the Elements of 2-categories}, Lecture Notes in Mathematics 420, Springer (1974)

\bibitem{borceux} F. Borceux, \textit{The Handbook of Categorical Algebra}, Cambridge University Press (1994)

\bibitem{elephant} P. T. Johnstone, \textit{Sketches of an Elephant: A Topos Theory Compendium}, Oxford University Press (2003)

\bibitem{ehr}
C. Ehresmann, \textit{Cat\'egories structur\'ees}, Annales de l'Ecole
Normale et Superieure 80 (1963)

\bibitem{bicat} J. B\'eanabou, \textit{Introduction to Bicategories}, Reports of the Midwest Category Seminar, Lecture Notes in Mathematics 47, Springer (1967)

\bibitem{jacobs}  B. Jacobs, \textit{Categorical Logic and Type Theory}, Elsevier (2001)


\bibitem{gag} M. F. Gouzou and R. Grunig, \textit{Fibrations Relatives}, Seminaire de Theorie des Categories (1976)

\bibitem{mrp}
M. Przybylek, \textit{Enriched vs. internal categories}, Master's thesies, University of Warsaw (2007)

\bibitem{mike} M. Shulman, \textit{Enriched indexed categories},  	arXiv:1212.3914 (2012)






\bibitem{phoa}
W. Phoa, \textit{An Introduction to Fibrations, Topos Theory, the Effective Topos and Modest Sets}, LFCS report ECS-LFCS-92-208 (1995)

\bibitem{plcat}
R. Seely, \textit{Categorical Semantics for Higher Order Polymorphic Lambda Calculus}, The Journal of Symbolic Logic, Volume 52, Issue 4 (1987)



\end{thebibliography}
\end{document}